\documentclass[reqno]{amsart}
\usepackage[margin=3.5cm]{geometry}
\usepackage{amsmath,amsfonts,amssymb, amsthm}
\usepackage{xcolor}
\usepackage{graphicx}
\usepackage{xypic}

\newtheorem{example}{Example}
\newtheorem{convention}{Convention}
\newtheorem{definition}{Definition}
\newtheorem{theorem}{Theorem}
\newtheorem{proposition}{Proposition}
\newtheorem{lemma}{Lemma}
\newtheorem{corollary}{Corollary}
\newtheorem{remark}{Remark}
\newtheorem*{nntheorem}{Theorem}

\newcommand{\Sol}{\operatorname{Sol}}

\newcommand{\trop}{\operatorname{trop}}
\newcommand{\ord}{\operatorname{ord}}

\title[Elementary solutions of OTDEs and vanishing orders]{Elementary solutions of ordinary tropical differential equations, and vanishing orders of  solutions of algebraic differential equations}
\author{Cristhian Garay-L\'opez}

\address{Centro de Investigaci\'on en Matem\'aticas, A.C. (CIMAT) Jalisco S/N, Col. Valenciana CP. 36023 Guanajuato, Gto, M\'exico}

\email{cristhian.garay@cimat.mx}

\author{Johana Luviano-Flores}

\address{Universidad Aut\'onoma Metropolitana-Azcapotzalco. \\Av. San Pablo 180, Col. Reynosa Tamaulipas, Alcald\'ia Azcapotzalco, CP. 02200 Mexico City, México
} 

\email{jlf@azc.uam.mx}

\author{Carla Valencia-Negrete}

\address{Department of Physics and Mathematics, Universidad Iberoamericana Ciudad de M\'{e}xico, Prolongacion Paseo de la Reforma 880, C. P. 01219, Mexico City, M\'exico}

\email{carla.valencia@ibero.mx}
\begin{document}
\maketitle

\begin{abstract}
Our aim is to use tropical differential algebra to systematically build a combinatorial basis for the study of the  set of (formal) power series solutions to nonlinear algebraic ordinary differential equations (over $\mathbb{C}$) expanded around the point $t_0=0\in\mathbb{C}$, which may also be effectively computed using standard tropical algebra. 
    
This paper is divided into two parts. First, given an ordinary  tropical differential equation in one differential variable $P=P(y)$ of (differential) order $k$,  we study the sets $Sol_{\mathbb{B}[\![t^{\Gamma}]\!],k}(P)\supset \mu(Sol_{\mathbb{B}[\![t^{\Gamma}]\!]}(P)\!)$ of tropical elementary  $k$-solutions and  minimal tropical solutions, respectively; we show that these two sets bear many similarities. We do this for tropical solutions $y=\varphi(t)\in \mathbb{B}[\![t^{\Gamma}]\!]$ (with  coefficients in the boolean semifield $\mathbb{B}$) of $P$ having support in different types of submonoids $\Gamma$ of $(\mathbb{R},+,0)$, namely formal ($\Gamma=\mathbb{N}$), Laurent ($\Gamma=\mathbb{Z}$) and Hahn ($\Gamma=\mathbb{R}$). 

Then, given an  ordinary algebraic differential equation $\mathfrak{P}$ (with meromorphic coefficients in one differential variable $\mathfrak{P}=\mathfrak{P}(y)$ and of differential order $k$), we consider the set  $S(\mathfrak{P},0):=ord_t(Sol_{\mathbb{C}[\![t^{\mathbb{R}}]\!]}(\mathfrak{P})\!)\subset\mathbb{R}$ of $t$-adic orders of formal Hahn  solutions of $\mathfrak{P}$, which is an algebraic object that gives  information about the nature of  the germs of solutions of $\mathfrak{P}$ at the point $t_0=0\in\mathbb{C}$. We show that this set  is contained in the set of $t$-adic orders of  Hahn elementary $k$-solutions of its tropicalization $P=trop(\mathfrak{P})$, this is
 $S(\mathfrak{P},0)\subset ord_t(Sol_{\mathbb{B}[\![t^{\mathbb{R}}]\!],k}(P))$. In most cases, the  set $Sol_{\mathbb{B}[\![t^{\mathbb{R}}]\!],k}(P)$ is a finite family of univariate tropical polynomials which indeed can be computed through standard  techniques of tropical algebra.   
\end{abstract}
\tableofcontents
\section{Introduction}

When attempting to find a power series solution to a nonlinear ordinary differential equation (abbr. \textsc{ode}) $\mathfrak{P}$ near a point of interest using traditional methods such as the Frobenius or Painlevé tests, usually it is necessary to guess a form, take derivatives, multiply polynomials, and carefully balance the smallest (lowest order) terms. Say that this point of interest is $t_0=0\in\mathbb{C}$. In Painlevé analysis, when starting a power series solution $\Phi=\Phi(t)$, we first guess a term, say, $c_ppt^p$, for which $p$ is an unknown power and $c_p\in\mathbb{C}^*$, then we evaluate this datum into $\mathfrak{P}$ and check which powers of $t$ appear. For the equation to be valid near $t_0=0$, the terms with the smallest exponents must cancel each other out; this is, at least two terms must have the same smallest exponent so their coefficients can cancel.

 This is the basic principle behind the formal concept of {\it tropical vanishing}, applied to an initial (or indicial) tropical equation constructed out of the pair $(\mathfrak{P},\Phi)$ using  classical valuation theory. Tropical mathematics yield a convenient algebraic language to handle these types of questions. In order to be able to apply the methods from Tropical Differential Algebraic Geometry, we need to  work with \textsc{ode}s $\mathfrak{P}$ which are algebraic (abbr. \textsc{oade}), these objects are special cases of {\it differential polynomials}, which are the elements of a differential ring $\mathfrak{P}\in\mathbb{C}[\![t]\!]\{y\}$ that indicates that  $\mathfrak{P}$ is a sum of differential monomials in a differential variable $y$, whose  coefficients  can be complex formal power series in $t$.

 Up to now, the application of tropical methods for \textsc{ode}s has been mostly focused on equations $\mathfrak{P}\in\mathbb{C}[\![t]\!]\{y\}$ and formal power series solutions $\Phi\in\mathbb{C}[\![t]\!]$ of them (which means that $\mathfrak{P}(\Phi)=0$ as an element of $\mathbb{C}[\![t]\!]$), see for example \cite{AGT16}, and \cite{GaMe} for the non-trivially valued case, but they can be extended to  equations $\mathfrak{P}\in\mathbb{C}(\!(t)\!)\{y\}$ and formal Hahn series solutions $\Phi\in\mathbb{C}[\![t^{\mathbb{R}}]\!]$, in fact, the elements of the set $Sol_{\mathbb{C}[\![t^{\mathbb{R}}]\!]}(\mathfrak{P})=\{\Phi\in \mathbb{C}[\![t^{\mathbb{R}}]\!]\::\:\mathfrak{P}(\Phi)=0\}$ can be seen as generalizations of germs of meromorphic (or Puiseux) functions at $t_0=0\in\mathbb{C}$ which are solutions to $\mathfrak{P}$, and there is a map 
    \begin{equation*}
        Supp:Sol_{\mathbb{C}[\![t^{\mathbb{R}}]\!]}(\mathfrak{P})\xrightarrow[]{}Sol_{\mathbb{B}[\![t^{\mathbb{R}}]\!]}(P),
    \end{equation*}
    where $P=trop(\mathfrak{P})$ is the tropicalization of $\mathfrak{P}$. See Prop. \ref{prop:trop_Hahn}. 

Let us describe more concretely the above terms. If $\mathfrak{P}\in\mathbb{C}(\!(t)\!)\{y\}$ is of the form
\begin{equation}
\label{eq:ODE}
\mathfrak{P}=\mathfrak{P}(y)={\mathcal A}_{M_1}E_{M_1}+\cdots+{\mathcal A}_{M_s}E_{M_s},\quad {\mathcal A}_{M_i}\in\mathbb{C}(\!(t)\!)^*,
\end{equation}
\noindent   with formal Laurent series coefficients ${\mathcal A}_{M_i}=t^{a_{M_i}} \sum_{j\geq 0} {\mathcal A}_{ij} t^j$, $a_{M_i}\in\mathbb{Z}$ and $\ {\mathcal A}_{i0}\neq 0$, then the {\it tropicalization} of $\mathfrak{P}$ is $P=trop(\mathfrak{P})=\bigoplus_i ord_t({\mathcal A}_{M_i})\odot E_{M_i}=\bigoplus_i{a}_{M_i}\odot E_{M_i}$. This is an example of an ordinary tropical differential equation $P=P(y)$ in one differential variable $y$ (abbr. \textsc{otde}), which is also  finite sum of (differential) monomials 
\begin{equation}
\label{eq:otde}
P=P(y)=t^{a_{M_1}}E_{M_1}+\cdots+t^{a_{M_s}}E_{M_s}=\bigoplus_i a_{M_i}\odot E_{M_i},\quad a_{M_i}\in\mathbb{Z}.
\end{equation}

This object has an associated set $Sol(P)\subset \mathcal{P}(\mathbb{N})$ of tropical solutions, which gets naturally ordered with respect to inclusion of sets; the set $\mu(Sol(P)\!)\subset Sol(P)$  of minimal (nonzero) tropical solutions can be considered  as some sort of {\it order-theoretic basis} for $Sol(P)$, since any tropical solution contains a minimal one. 

In \cite{GG26}, the authors started the study of the set  $\mu(Sol(P)\!)$  for the linear case\footnote{Linear homogeneous in $n\geq1$ differential variables $y_1,\ldots,y_n$.}. The key insight was the decomposition
\begin{equation}
\label{eq:decomposition}
\mu(Sol(P)\!)=F(P)\sqcup\mathcal{C}_{\mathbb{Z}_{\geq0}}(P),
\end{equation}
where $\mathcal{C}_{\mathbb{Z}_{\geq0}}(P)$ consists of  the   solutions having ($t$-adic) order at least the differential order of $P$. We can talk about the $t$-adic order of a tropical solution, since the usual $t$-adic valuation $ord_t:\mathbb{C}[\![t]\!]\xrightarrow[]{}\mathbb{N}\cup\{\infty\}$ of formal power series factors through the support map $Supp:\mathbb{C}[\![t]\!]\xrightarrow[]{}\mathcal{P}(\mathbb{N})$ using a map $\mathcal{P}(\mathbb{N})\xrightarrow[]{}\mathbb{N}\cup\{\infty\}$ (also denoted by $ord_t$). See diagram \eqref{eq:ord_comm_with_supp}.

This paper consists of two parts. The first one is the introduction and study of the set $Sol_k(P)$ of elementary $k$-solutions of an \textsc{otde}  $P$  as in \eqref{eq:otde}, where $k$ is the differential order of $P$. These solutions are univariate polynomials, and they showcase the tropical phenomenon that only a finite amount of information is necessary to determine the tropical solutions of $P$. Moreover, there is an inclusion $\mu(Sol(P))\subset Sol_k(P)$, and the set $Sol_k(P)$ can be studied with similar methods, since there is a corresponding  decomposition $Sol_k(P)=Sol_k(P,[k-1])\sqcup Sol_k(P,[k-1]^c)$ compatible with \eqref{eq:decomposition}.

Afterwards we give a complete characterization of the space of \textsc{otde}s $P$ satisfying $\#Sol_k(P)=\#\infty$ within the space of parameters $\mathbb{Z}_S$ of \textsc{otde}s \eqref{eq:otde} with a fixed support $S$; namely, following \cite{GG26}, we describe the set 
\begin{equation}
\label{eq:sing}
Sing(S)=    \{P\in\mathbb{Z}_S\::\:\#Sol_k(P)=\infty\}.
\end{equation}
It is worth mentioning that, once we fix a support set $S$ of order $k$, the family  $\{Sol_k(P)\::\:P\in \mathbb{Z}_S\}$ can be organized  in a bundle $Sol_k(P)\xrightarrow[]{}I_S\xrightarrow[]{\pi_1}\mathbb{Z}_S$, where the corresponding incidence variety $I_S= \{(P,\varphi)\in \mathbb{Z}_S\times \pi_k(\mathbb{B}[\![t]\!])\::\:\varphi\in Sol_k(P)\}$  is {\it tropical semi-algebraic}, and it can be studied from a geometric point of view through a tropical hypersurface $V(\bigodot_{R\in\mathcal{P}(k-1)}\Delta_{S,R})\subset (\mathbb{T}^*)^{\#S}\times \mathbb{T}$ using the family of tropical polynomials $\{\Delta_{S,R}\::\:R\in\mathcal{P}([k-1])\}$ from \eqref{eq:disc}. See Section \ref{sect:discriminants} for more details.

The second part of this paper is motivated by the application of tropical methods to the analytic theory of \textsc{oade}s $\mathfrak{P}$ as in \eqref{eq:ODE}, which naturally leads to the consideration of  tropical solutions to \textsc{otde}s having more general supports. In the  theory of differential equations, a singularity of $\mathfrak{P}$ is a point $t_0\in\mathbb{C}$ where a solution is no longer analytic \cite{Ahl,AF,BCh}.  The problem that Painlevé addressed in 1900 was as follows: ``To determine all algebraic differential equations  of the first, second, and third orders, and so on, whose general integral is uniform.'' \cite{P00}. These are differential equations whose solutions are single-valued, globally defined functions \cite{P00,G12,CM}. The \emph{Painlevé property} of a  differential equation of order $k$ is governed by the behavior of the solutions around \emph{movable singularities} that can only be poles \cite{RGB,JK,FK}.

In this paper we will focus rather on the non-linear \textsc{ode}s, since the linear case has been addressed in \cite{F25} for the case of ADEs, and in \cite{GG26} for the case of \textsc{otde}s (homogeneous in one variable). In general, for the non-linear case, movable singularities can manifest as ordinary poles, essential singularities, or branch points (which can be algebraic or logarithmic) \cite{FK}. Therefore, for an integral to remain ``uniform'' in Pailevé's sense, any singularity $t_0\in \mathbb{C}$ that changes its position when the initial conditions are changed must be strictly forbidden from being a branch point; it can only be an ordinary pole. This is the \emph{strong Painlevé property}, which states that movable singularities are, at most, ordinary poles, which can be accurately represented by Laurent series with integer exponents. In contrast, many physical nonlinear systems exhibit algebraic branch points or fractional resonances. \textsc{ode}s defined in $\mathbb{C}$ with movable algebraic singularities and solutions that can be presented locally as convergent Puiseux series expansions are considered a generalization of the Painlevé property \cite{FK}.

In Section \ref{Sec:fromCtoT} we explain why  formal Laurent solutions $\Phi\in\mathbb{C}(\!(t)\!)=\mathbb{C}[\![t^{\mathbb{Z}}]\!]$ of an \textsc{oade} $\mathfrak{P}$ are an algebraic generalization of germs of solutions $\mathfrak{P}(\Phi)=0$ with $\Phi$ meromorphic at $t_0=0\in\mathbb{C}$, and we show that the tropical formalism extends in such a way that if $\mathfrak{P}(\Phi)=0$ with $\Phi\in\mathbb{C}[\![t^{\mathbb{Z}}]\!]$, then  $Supp(\Phi)\in\mathcal{P}(\mathbb{Z})$ is  a tropical Laurent solution to $trop(\mathfrak{P})$. This type of tropical solution was already proposed by D. Grigoriev in \cite{DG}. In a similar fashion, we can consider the set of   Laurent elementary $k$-solutions $LSol_k(P)\subset \mathcal{P}(\mathbb{Z})$. For applications to  \textsc{oade}s, we will be interested in the case in which $LSol_k(P)$ is finite, so for a fixed support $S$,  we compute the set  

\begin{equation}
\label{eq:sing_laurent}
LSing(S)=    \{P\in\mathbb{Z}_S\::\:\#LSol_k(P)=\infty\}.
\end{equation}

After that, we go all the way up to consider formal Hahn solutions $\Phi\in\mathbb{C}[\![t^{\mathbb{R}}]\!]$ to $\mathfrak{P}$, which  encompass Puiseux series with rational exponents (developed around $t_0=0\in\mathbb{C}$) as a subset and can therefore accommodate the non-integer and fractional behaviors necessary for the weak Painlevé test at $t_0=0\in\mathbb{C}$, where the solution is allowed to have a finite number of branches \cite{CM}. This is how we come to propose the set 
\begin{equation}
\label{eq:OCDE_tadicord}
    S(\mathfrak{P},0):=ord_t(\{\Phi(t)\in\mathbb{C}[\![t^{\mathbb{R}}]\!]\::\:\mathfrak{P}(\Phi(t)\!)=0\})\subset \mathbb{R}
\end{equation}
for  an \textsc{oade} $\mathfrak{P}$ as an algebraic object that gives  information about the nature of  the germs of solutions of $\mathfrak{P}$ at the point $t_0=0\in\mathbb{C}$. We show the inclusion $S(\mathfrak{P},0)\subset ord_t(HSol_k(trop(\mathfrak{P})))$, where the right-hand side is the set of $t$-adic orders of  tropical Hahn elementary $k$-solutions of the tropical differential equation $trop(\mathfrak{P})$, which can be computed using standard tropical algebra methods. We also study the locus of differential polynomials $P\in\mathbb{Z}_S$ for which the set $HSol_k(P)$ is finite in Sections \ref{sec:ths} and \ref{sec:ordatleastk}.

In the non-archimedean   direction, \cite[Corollary 5.3]{GaMe} found that if we replace our Banach field $K=\mathbb{C}$ with an algebraically closed field  of characteristic 0, for which there is a diagram $\xymatrix{v:K\ar@/^/[r]&\mathbb{T}:s\ar@/^/[l]}$
where $v$ is a non-archimedean valuation of rank 1 (i.e. non trivially valued) with a section $s$, e.g., the non-Archimedean Banach field $\mathbb{C}_p$ of the $p$-adic complex numbers. Then the (non-archimedean) radius of convergence of formal power series solutions to systems $\Sigma\subset K[\![t]\!]\{y\}$ can also be found tropically. 

\subsection{Results}
For the first part of the paper. In Theorem \ref{prop:ot_basis_Hahn}, we show that any tropical solution $\varphi\in Sol(P)$ contains both an elementary $k$-solution $\pi_k(\varphi)$ and (at least) a minimal solution. We do this for the case of Hahn tropical solutions, which is the most general case.

Afterwards, for $S=\{M_1,\ldots,M_s\}\subset \mathbb{N}^{k}$ a fixed support set of (differential) order $k$, we consider the parameter space $\mathbb{Z}_S\subset(\mathbb{T}^*)^{\#S}$ of equations \eqref{eq:otde} having support $S$. In Theorem \ref{thm:crit_fin}, we show that $\#Sol_k(P)<\infty$ (and thus $\#\mu(Sol(P)\!)<\infty$) for $P\in\mathbb{Z}_S$ if and only if $P$ is regular, moreover, we describe the locus $X(S)\subset \mathbb{Z}_S$ of non-regular polynomials (those $P$ satisfying $\#Sol_k(P)=\infty$). 

\begin{nntheorem}[See Theorem \ref{thm:crit_fin}]
     If $P\in\mathbb{Z}_S$, then  $\#Sol_k(P)<\infty$ if and only if $P$ is regular.    
\end{nntheorem}

In Theorem \ref{thm:crit_fin_mero} we show that $\#LSol_k(P)<\infty$ (and thus $\#\mu(LSol(P)\!)<\infty$) for $P\in\mathbb{Z}_S$ if and only if $P$ is Laurent regular, moreover, we find the description of the locus $X_L(S)$ of polynomials $P$ satisfying $\#LSol_k(P)=\infty$.

\begin{nntheorem}[See Theorem \ref{thm:crit_fin_mero}]
     If $P\in\mathbb{Z}_S$, then  $\#Sol_k(P)<\infty$ if and only if $P$ is Laurent regular. 
\end{nntheorem}

The second part of the paper starts at Section \ref{Sec:App}. We introduce the set $S(\mathfrak{P},0)\subset \mathbb{R}$ from \eqref{eq:OCDE_tadicord} for an \textsc{oade} $\mathfrak{P}$. The next result shows that $S(\mathfrak{P},0)$ can be computed from the set elementary Hahn $k$-solutions of the tropicalization $P=trop(\mathfrak{P})$.
 
\begin{nntheorem}[See Theorem \ref{thm:app_tropical_Painleve}]
    If $\mathfrak{P}$ is an \textsc{oade} and $P=trop(\mathfrak{P})$, then $S(\mathfrak{P},0)\subset ord_t(HSol_k(P))$.
\end{nntheorem}

Throughout all the paper, we  divide our support sets $S$ (see Def. \ref{def:supsetk}) in homogeneous\footnote{if all the monomials  of $S$ have the same degree.} of degree $d>1$ (a case which is very similar to the linear one), and $S$ is not homogeneous, in which case we can express $P\in\mathbb{Z}_S$ as the sum of its homogeneous parts: $P=P_{d_1}+\cdots+P_{d_m}$, for some $d_1<\cdots<d_m$, $m>0$. The sets $S_{min}$ and $S_{max}$ of minimal and maximal degree monomials in the support of $P$ will be crucial for our analysis.

\subsection{Roadmap}
In Section \ref{sect:preliminaries} we give some standard preliminaries on tropical algebra and tropical differential algebra.  After that, the paper is divided in two rather independent parts, the first one consists of sections \ref{sec:Gen_prop_Min_Sol} and \ref{sect:F(P)}, and it is mainly an extension  of the analysis of minimal tropical solutions of \textsc{otde}s from \cite{GG26} to the non-linear case. 

The technical core of the paper is Section \ref{sec:Gen_prop_Min_Sol}, and it represents the basis for studying tropical solutions over more general monoids $\mathbb{N}\subseteq\Gamma\subseteq\mathbb{R}$. Here we introduce the set $Sol_k(P)$ of elementary $k$-solutions of $P$, for $k\in\mathbb{Z}_{>0}$ the differential order of $P$, and we show that there is a map $\pi_k:Sol(P)\xrightarrow[]{}Sol_k(P)$. Then in Section \ref{sect:discriminants} we introduce the concept of support set in Definition \ref{def:supsetk}, and introduce a family of tropical polynomials 

    \begin{equation}
    \label{eq:disc}
        \Delta_{S,R}=\bigoplus_{M\in S}|M|_R(z)\odot x_M\in\mathbb{T}[x_M\::\:M\in S,z], \quad R\in\mathcal{P}([k-1]),
\end{equation}
defined in the space $\mathbb{Z}_S\times\overline{\mathbb{Z}}$, where $\mathbb{Z}_S$ is the space of parameters of all the \textsc{otde}s having support $S$.  The  set $Sol_k(P)$ is the disjoint union of two proper subsets, namely $Sol_k(P,[k-1]^c)$ consisting of those whose $t$-adic order does not lie on $\{0,\ldots,k-1\}$ --this is studied in Sect. \ref{Sec:ordgeqk}--, and $Sol_k(P,[k-1])$ consisting  of those whose $t$-adic order lies on $\{0,\ldots,k-1\}$; this is studied in Sect. \ref{sect:F(P)} using the function 

\begin{equation}
\label{eq:disc:specialization}
\Delta_S|_P(z)=\Delta_S(a_{M_1},\ldots,a_{M_s},z):=min_j\{|M_j|_S+a_{M_j}+deg(M_j)z\}.
\end{equation}
Then we prove Theorem \ref{thm:crit_fin} in Sect. \ref{sec.generic_solutions}.

The second part  consists of Section \ref{Sec:App}, and is about potential applications of the first part towards the theory of algebraic ordinary differential equations over the complex numbers. 

In Section \ref{Sec:fromCtoT} we explain how to connect the language of meromorphic solutions of \textsc{oade}s $\mathfrak{P}$ with meromorphic coefficients with the language of tropical differential equations and their tropical (Laurent) solutions. We also introduce the important set $S(\mathfrak{P},0)\subset \mathbb{R}$ from \eqref{eq:OCDE_tadicord} of $t$-adic orders of (germs of)  formal Hahn power series solutions of an algebraic differential equation $\mathfrak{P}$ at the point  $t_0=0\in\mathbb{C}$.

Then we extend the set $Sol_k(P)$ to the set $Sol_{\mathbb{B}[\![t^{\mathbb{Z}}]\!],k}(P)$ of Laurent elementary $k$-solutions. Theorem \ref{thm:crit_fin_mero} is proved in Sect. \ref{sec:trop_LS}. Then starting in Sect. \ref{sec:ths}, we start the study of the set $Sol_{\mathbb{B}[\![t^{\mathbb{R}}]\!],k}(P)$ Hahn elementary $k$-solutions, which embraces the Puiseux case, and where Theorem \ref{prop:ot_basis_Hahn} is proved. To finish, in Sect. \ref{sect:appPPord}, we return the focus  to the set $S(\mathfrak{P},0)$, and we prove Theorem \ref{thm:app_tropical_Painleve}.

\subsection{General conventions}
If $X$ is a set we denote by $\mathcal{P}(X)$ its power set and by $\#X$ its cardinality. Disjoint union is denoted by $A\sqcup B$, and if $Y\subset X$, the relative difference is denoted $X\setminus Y.$

If $(P,\leq)$ is a poset and $x\in P$, we denote by $\langle x\rangle^{\uparrow}=P_{\geq x}=\{y\in P\::\:x\leq y\}$ and $\langle x\rangle^{\downarrow}=P_{\leq x}=\{y\in P\::\:y\leq x\}$. If $(P,\leq)$ is totally ordered, we denote by $\mathcal{P}_{wo}(P)$ the family of well-ordered subsets of $P$.

We consider $\mathbb{R}$ equipped with the usual ordering, and $\mathbb{N}:=\mathbb{Z}_{\geq0}$. If $k\in \mathbb{N}$, we denote by $[k]=\{0,\ldots,k\}$, and  by $\mathbb{N}^{[k]}$ the free $\mathbb{N}$-semimodule $\mathbb{N}^{\oplus k+1}$ with canonical basis $\{e_0,\ldots,e_k\}$, this is, $\mathbb{N}^{[k]}=\bigoplus_{j=0}^k\mathbb{N}\cdot e_j$.

We will denote by $\mathfrak{P}\in\mathbb{C}(\!(t)\!)\{y\}$ a differential polynomial, which generalizes \textsc{oade} with meromorphic coefficients, and by $P=trop(\mathfrak{P})$ the \textsc{otde} that results from tropicalizing it. In the sequel, when talking about tropical  solutions  of \eqref{eq:otde}, we omit the adjective ``tropical'' for brevity.
\section{Preliminaries}
\label{sect:preliminaries}
In this Section we discuss the basic facts from tropical algebraic geometry; a standard source for this Section is \cite{MLBook}.
We denote by $\mathbb{T}=(\mathbb{R}\cup\{\infty\},\min=\oplus,+=\odot)$ the tropical semifield, with  $0_{\mathbb{T}}=\infty$ and $1_{\mathbb{T}}=0$.

\begin{remark}
\label{rem:notation_trop_van}
A sum $a=a_1\oplus\cdots\oplus a_s$ with $a_i\in \mathbb{T}$ vanishes tropically if either
\begin{enumerate}
    \item $a=0_{\mathbb{T}}=\infty$, which happens if and only if $a_i=0_{\mathbb{T}}=\infty$ for all $i$, or 
    \item  $a\neq0_{\mathbb{T}}=\infty$, and  there exists $1\leq i\neq j\leq s$ such that $a=a_i=a_j$.
\end{enumerate}
In order to increase the similarity  of  condition (2) above with the usual algebraic notation, sometimes we will denote it as  $a=a_1\oplus\cdots\oplus a_s\!\!\text{``}\!=\!\!\!\text{''}0_{\mathbb{T}}$.
\end{remark}

The condition $a=a_1\oplus\cdots\oplus a_s\!\!\text{``}\!=\!\!\!\text{''}0_{\mathbb{T}}$ is based in the following principle. Being idempotent, terms cannot be summed to zero through addition in $\mathbb{T}$, rather than subtraction, canceling terms is conceptually equivalent to tying. 

We denote by $\mathbb{T}[x_1,\ldots,x_n]$ the semiring of tropical polynomials in  $n\geq1$ variables $x=(x_1,\ldots,x_n)$.   When possible, we denote tropical polynomials using tropical symbols  $f=\bigoplus_{I\in S} a_I\odot x^{I}$ with $a_I\in\mathbb{T}$, and we denote by $ev_f$ the evaluation map $\mathbb{T}^n\xrightarrow[]{}\mathbb{T}$  induced by $f$, this is, if $a=(a_1,\ldots,a_n)\in \mathbb{T}^n$, then 
\begin{equation}
    ev_f(a):=\min_{I\in S}\{a_I+\langle I,a\rangle\}.
\end{equation}

If $f=\bigoplus_{I\in S} a_I\odot x^{\odot I}$, we will assume that $S\subset \mathbb{N}^n$ is the support of $f$, this is, $S=Supp(f)=\{I\in \mathbb{N}^n\::\:a_I\neq 0_{\mathbb{T}}\}$ for $I\in S$. 

We denote by $V(f)\subset \mathbb{T}^n$   the tropical hypersurface  defined by $f$, namely \begin{equation*}
    V(f)=\{a=(a_1,\ldots,a_n)\in \mathbb{T}^n\::\:ev_f(a)\text{ vanishes tropically}\}.
\end{equation*}

\begin{definition}
    A tropical prevariety  is a finite intersection of tropical hypersurfaces. If $\{f_1,\ldots,f_m\}\subset \mathbb{T}[x_1,\ldots,x_n]$, we denote by $V(f_1,\ldots,f_m)=\bigcap_{i=1}^m V(f_i)$ the associated tropical prevariety.
\end{definition}

Tropical prevarieties are in general not tropical varieties in the traditional sense (see \cite[\S3.2]{MLBook}), but they are (rational) polyhedral sets in $\mathbb{T}^n$

We denote by $\mathbb{T}^*=(\mathbb{R},\odot,0)$ the tropical torus. The tropical hypersurface $V(f)$ consists of two parts: $\{a\in\mathbb{T}^n\::\:ev_f(a)=0_{\mathbb{T}}\}\subset \mathbb{T}^n\setminus (\mathbb{T}^*)^n$, and $\{a\in\mathbb{T}^n\::\:ev_f(a)\!\text{``}\!=\!\!\!\text{''}0_{\mathbb{T}}\}$, which is sometimes referred to as the bend locus of $f$. 

The bend locus of $f$  refers to the set of points where this minimum is attained simultaneously by at least two distinct monomials. If a single monomial attains the minimum, it strictly dominates. In contrast, if two monomials achieve the minimum, they are considered balanced. Thus the bend locus of $f$ is empty unless $f$ has at least two monomials, something which we will assume. Note that  $V(f)\cap (\mathbb{T}^*)^n$ is contained in the  bend locus of $f$. 

\begin{definition}
\label{def:set_of_exponents}
    If  $ev_f(a)\!\!\text{``}\!=\!\!\!\text{''}0_{\mathbb{T}}$,  we denote by $S(a)$ the set of monomials where $ev_f(a)$ reaches its minimum, this is $S(a):=\{I\in S\::\:ev_f(a)=a_I+\langle I,a\rangle\}$.
\end{definition}

\begin{convention}
\label{conv:polynomials}
In general we will try to use different notation for a tropical polynomial $f=\bigoplus_{I\in S} a_I\odot x^{I}$ and for its induced evaluation map $ev_f(x)=min_{I\in S}\{a_I+\langle I,x\rangle\}$, where $\mathbb{T}^n$ is endowed with coordinates $\{x_1,\ldots,x_n\}$. This distinction has the purpose of making a smooth transition between the concepts and notation from commutative algebra and the actual computations using tropical operations.
\end{convention}

\begin{definition}
\label{def:rel_mon}
    A monomial $J\in S$ of a tropical polynomial $f$ is irrelevant if $ev_f(a):=min_{I\in S\setminus\{J\}}\{a_I+\langle I,a\rangle\},$ otherwise we say that it is relevant.
\end{definition}

The binary relation on $\mathbb{T}[x_1,\ldots,x_n]$ defined by $f\sim g$ if and only if $ev_f=ev_g$ is a semiring congruence; we denote by $\mathcal{O}_{\mathbb{T}^n}(\mathbb{T}^n)$ the quotient semiring, which is the semiring of global tropical functions on $\mathbb{T}^n$.

If $\Gamma\subset\mathbb{R}$ is an additive monoid, and $R$ is a commutative semiring with 1, the Mal'cev-Neumann semiring $R[\![t^\Gamma]\!]$ consists of the set of formal sums $\varphi=\sum_{m\in \Gamma}a_mt^m$ with $a_m\in R$ such that $Supp(\varphi)=\{m\in \Gamma\::\:a_m\neq0\}$ is a well-ordered subset of $\Gamma$. This defines a map $Supp:R[\![t^\Gamma]\!]\xrightarrow[]{}\mathcal{P}_{wo}(\Gamma)$.

We define $\overline{\Gamma}=\Gamma\cup\{\infty\}$, which is a sub-semiring of $ \mathbb{T}$, and $ord_t:R[\![t^\Gamma]\!]\xrightarrow{}\overline{\Gamma}$ by $ord_t(0)=\infty$ and $ord_t(\varphi)=min Supp(\varphi)$.

\begin{example}
\label{ex:BPS}
    We denote by $\{0,\infty\}=\mathbb{B}\subset\mathbb{T}$ the boolean semifield. Sometimes we write $1_{\mathbb{B}}=0$ and $0_{\mathbb{B}}=\infty$. The support map $Supp:\mathbb{B}[\![t^\Gamma]\!]\xrightarrow[]{}\mathcal{P}_{wo}(\Gamma)$ is an isomorphism of semirings, so we can express the elements of $\mathbb{B}[\![t^\Gamma]\!]$ either 
\begin{enumerate}
    \item as formal sums $\varphi=\sum_{i\in A}1_{\mathbb{B}}t^i$. Sum and product  of series are performed with the usual definitions,  or 
    \item as sets $A\in\mathcal{P}_{wo}(\Gamma)$. Under this presentation, sum of series becomes union, and product of series becomes Minkowski sum.
\end{enumerate}
In particular, the null series 0 has support $Supp(0)=\emptyset$, and the series $1=t^0$ has support $Supp(1)=\{0\}$. The inverse isomorphism is $e:\mathcal{P}_{wo}(\Gamma)\xrightarrow[]{}\mathbb{B}[\![t^\Gamma]\!]$, which sends $A$ to $e^A=\sum_{i\in A}1_{\mathbb{B}}t^i$. This satisfies $e^{\emptyset}=0$ and $e^0=1$.
\end{example}

 If $K$ is a commutative ring with 1, we denote by $v_0:K\xrightarrow[]{}\mathbb{B}$ the trivial valuation, and we denote also by $v_0$ the induced map $v_0:K[\![t^\Gamma]\!]\xrightarrow[]{}\mathbb{B}[\![t^\Gamma]\!]$ sending $\varphi=\sum_{m\in M}a_mt^m$ to $v_0(\varphi)=\sum_{m\in M}v_0(a_m)t^m$. Note that this is just the support map in disguise, and that the following diagram commutes 
 \begin{equation}
 \label{eq:ord_comm_with_supp}
     \xymatrix{K[\![t^\Gamma]\!]\ar[r]^{Supp}\ar[dr]_{ord_t}&\mathbb{B}[\![t^\Gamma]\!]\ar[d]^{ord_t}\\
    &\overline{\Gamma}.}
 \end{equation}

\subsection{Tropical differential algebra}
\label{sect:TDA}
We give basic definitions from tropical differential algebra adapted for this particular case. In this introduction we follow the notation from Convention \ref{conv:polynomials} with the purpose of making a smooth transition between the objects from differential equations, and the actual computations using tropical mathematics. The source for this material is \cite{CGL}.

A differential monomial in one differential variable $y$ is a monomial in the variables $\{y_i=y^{(i)}\::\:i\in\mathbb{N}\}$, this is, a finite expression of the form 
\begin{equation}
\label{eq:diffmon}
E=\prod_{i\in\mathbb{N}}y_i^{m_i},\:m_i\in\mathbb{N},\:m_i=0\text{ for almost all }i.
\end{equation}

in order to lighten up the notation, it is better to use sub-indexes $y_i=y^{(i)}$ to denote the order of derivation. We can identify \eqref{eq:diffmon} with the vector $(m_0,m_1,\ldots)=M\in\mathbb{N}^{\oplus\mathbb{N}}$ and write $E=E_M$. The differential order of $E_M$ is $ord(E_M)=\max\{i\::\:m_i\neq0\}$, so $M\in \mathbb{N}^{ord(E_M)}$.

An ordinary tropical differential equation in one variable is a finite sum $P=P(y)=\sum_{M_i\in S}t^{a_{M_i}}E_{M_i}=\bigoplus_{M_i\in S} a_{M_i}\odot E_{M_i}$ as in \eqref{eq:otde} consisting of differential monomials \eqref{eq:diffmon}. Recall that the support of $P$ is  $Supp(P)=\{M_1,\ldots,M_s\}\subset \mathbb{N}^{\oplus\mathbb{N}}$, and the differential order of $P$ is $ord(P)=\max\{ord(E_M)\::\:M\in Supp(P)\}$.

\begin{definition}
    We denote by $\mathbb{B}[\![t]\!]=\mathbb{B}[\![t^{\mathbb{N}}]\!]$ the semiring of formal power series with coefficients in $\mathbb{B}$.  We denote by  $\tfrac{d}{dt}:\mathbb{B}[\![t]\!]\xrightarrow[]{}\mathbb{B}[\![t]\!]$ the formal derivation of series.
\end{definition}

The isomorphism $Supp:\mathbb{B}[\![t]\!]\xrightarrow[]{}\mathcal{P}(\mathbb{N})$ from Example \ref{ex:BPS} can be promoted to an isomorphism of pairs $(R,d)$ of semirings $R$ equipped with a derivation $d$. By Example \ref{ex:BPS}, we can express the elements of $\mathbb{B}[\![t]\!]$ either 
\begin{enumerate}
    \item as formal sums $\varphi=\sum_{i\in A}1_{\mathbb{B}}t^i$, where the derivation of series is performed with the usual conventions,  or 
    \item as sets $A\in\mathcal{P}(\mathbb{N})$.
\end{enumerate}

As for the case of tropical polynomials, a differential polynomial $P$ as in \eqref{eq:otde} induces a map  $dev_P:\mathbb{B}[\![t]\!]\xrightarrow[]{}\mathbb{B}[\![t]\!]$, called the differential evaluation, as follows:  if $P=E$ is a single monomial \eqref{eq:diffmon} and $\varphi\in \mathbb{B}[\![t]\!]$, then $dev_E(\varphi)$ is defined by replacing each instance of $y_i$ for $y_i(\varphi)=\tfrac{d^i\varphi}{dt^i}$:
\begin{equation*}
   E(\varphi)= dev_E(\varphi):=\prod_{i\in\mathbb{N}}(\tfrac{d^i\varphi}{dt^i})^{m_i},
\end{equation*}
and then we extend by linearity: $P(\varphi):=dev_P(\varphi)=\sum_it^{a_{M_i}}E_{M_i}(\varphi).$

If $\overline{\mathbb{N}}=\mathbb{N}\cup\{\infty\}\subset\mathbb{T}$,  we denote by $ord_t:\mathbb{B}[\![t]\!]\xrightarrow[]{}\overline{\mathbb{N}}$ is the usual $t$-adic valuation homomorphism. We have a commutative diagram: 
\begin{equation*}
    \xymatrix{\mathbb{B}[\![t]\!]\ar[r]^{dev_P}\ar[dr]_{trop_P}&\mathbb{B}[\![t]\!]\ar[d]^{ord_t}\\
    &\overline{\mathbb{N}}.}
\end{equation*}   

We say that $\varphi\in\mathbb{B}[\![t]\!]$ is a tropical solution of $P$ if the following expression vanishes tropically (see Remark \ref{rem:notation_trop_van}):
\begin{equation}
\label{eq:solution}
    \begin{aligned}
        ord_t(P(\varphi)\!)&=ord_t(\sum_it^{a_{M_i}}E_{M_i}(\varphi)\!)=\bigoplus_{i=1}^s a_i\odot ord_t(E_{M_i}(\varphi)\!)\\
        &=\min_j\{a_j+ ord_t(E_{M_j}(\varphi)\!)\}.
    \end{aligned}
\end{equation}

Continuing with our conventions, we will try to use different notation for a differential polynomial $P=\sum_{M_i\in S}t^{a_{M_i}}E_{M_i}=\bigoplus_{M_i\in S} a_{M_i}\odot E_{M_i}$ and for its induced evaluation map $ord_t(P(\varphi)\!)=\min_{M_i\in S}\{a_{M_i}+ ord_t(E_{M_i}(\varphi)\!)\}$.
 
If $\varphi\in\mathbb{B}[\![t]\!]$ is a solution of $P$, then by Remark \ref{rem:notation_trop_van} we have two possibilities for the tropical vanishing of \eqref{eq:solution}:
\begin{enumerate}
    \item $ord_t(P(\varphi)\!)=0_{\mathbb{T}}$, or
    \item $ord_t(P(\varphi)\!)\!\text{``}\!=\!\!\text{''}0_{\mathbb{T}}$, which appears only when $P$ is not a monomial. 
\end{enumerate}

For $M\in \mathbb{N}^{[k]}$, sometimes we write $M=\sum_{i=0}m_{i}e_i$ and $deg(M)=\sum_im_i$. If $P=\sum_{M_i\in S}t^{a_{M_i}}E_{M_i}$ has no constant term, then $\min ord(M):=\min\{i\::\:m_{i}\neq0\}\geq0$ for every $M$ in its support (since $M\neq0$), so $\min ord(P):=\min_{j}\{\min ord(M_j)\}\geq0$. We have the following result.

\begin{proposition}
\label{prop:equaltozerosols}
    If $P=\sum_{M_i\in S}t^{a_{M_i}}E_{M_i}$ has no constant term, then the set of tropical solutions $\varphi$ of $P$ satisfying $ord_t(P(\varphi)\!)=0_{\mathbb{T}}$ is the set of polynomials in $t$ with coefficients in $\mathbb{B}$ having degree at most $\min ord(P)-1$.
\end{proposition}
\begin{proof}
    Again, $ord_t(P(\varphi)\!)=0_{\mathbb{T}}$ happens if and only if $P(\varphi)=0$ (the null series) if and only if $E_{M_i}(\varphi)=0$ for all $i=1,\ldots,s$. And $E_{M}(\varphi)=0$ iff $deg(\varphi)\leq \min ord(E)-1$. 
\end{proof}

\begin{convention}
\label{conv:shape_of_P}
From now on our  differential polynomials $P=\sum_{i=1}^st^{a_{M_i}}E_{M_i}$ will satisfy: 
\begin{enumerate}
    \item have no constant term (this is relevant for our applications), 
    \item they are not a monomial ($s>1)$.
    \item they are differential equations, i.e. $ord(P)=k\in\mathbb{Z}_{>0}$.
\end{enumerate}
Since we already computed in Prop. \ref{prop:equaltozerosols} the tropical solutions satisfying $ord_t(P(\varphi)\!)= 0_{\mathbb{T}}$, from now on we will be interested only on the tropical solutions satisfying $ord_t(P(\varphi)\!)\text{``} = \text{'' } 0_{\mathbb{T}}$ (in particular, $0\notin Sol(P)$). We will also drop the adjective {\it tropical} in solution, and  refer to them only as solutions.
\end{convention}
\begin{definition}
    For $P=P(y)=\sum_{M_i\in S}t^{a_{M_i}}E_{M_i}$ as in \eqref{eq:otde}, we denote by 
     \begin{equation*}
         Sol(P)=\{\varphi\in\mathbb{B}[\![t]\!]\::\:ord_t(P(\varphi)\!)\text{``} = \text{'' } 0_{\mathbb{T}}\}
     \end{equation*}
endowed with the order induced by inclusion of subsets. 
\end{definition}

\begin{example}[The linear case]
\label{ex:lin_case}
    For us, an \textsc{otde} $P$ is linear of order $k$ if $Supp(P)=\{e_0,\ldots,e_k\}$, this is, if $P=\bigoplus_{i=0}^k a_{i}\odot y_i$ (in particular, it is homogeneous  and generic). In this case, $Sol(P)$ is a $\mathbb{B}$-semimodule, and the set $\mu(Sol(P)\!)$ consists of additive-irreducible elements which gives  some sort of order-theoretic basis.
\end{example}

\begin{remark}
    It is possible to have $Sol(P)=\emptyset$, see \cite[Theorem 3]{GG26} for a description of the case in which $P$ is linear. 
\end{remark}

An important subset of $Sol(P)$ is its set $\mu(Sol(P)\!)$ of minimal elements, which was introduced in \cite{GG26} for $P$ linear (homogeneous, in $n\geq1$ differential variables) and also for systems $\Sigma=\{P_1,\ldots,P_n\}$ of $n$ linear polynomials in the same number of variables. 

In the next section we will see that for $P$ general, it is better to focus in another subset of $Sol(P)$, which is its set $Sol_k(P)$ of elementary $k$-solutions, where $k$ is the differential order of $P$, and for this  we shall focus on the set $S=Supp(P)$ rather than on $P$ itself.

\section{General properties of elementary $k$-solutions}
\label{sec:Gen_prop_Min_Sol}
Let $k\in\mathbb{Z}_{>0}$.   In the  case or a linear differential polynomial $P$ (see Example \ref{ex:lin_case}) the particular structure of the support set  makes possible to bound the number of monomials of every minimal solution, this is one of the reasons why  elementary $k$-series were not considered in that case.

In this Section, we introduce the set of elementary $k$-series $\pi_k(\mathbb{B}[\![t]\!])\subset \mathbb{B}[\![t]\!]$, which are important in the tropical theory and have many properties, for instance, if $\mathbb{Z}_S$ denotes the parameter space of \textsc{otde}s having support $S$ of differential order $k$, and $Sol_k(P)$ is the set of elementary $k$-series which are solutions of $P\in \mathbb{Z}_S$, then  the  elements of $Sol_k(P)$ are polynomials,  and we have $\mu(Sol(P)\!)\subset Sol_k(P)$. Besides there is a {\it bundle} presentation $$Sol_k(P)\xrightarrow[]{}I_S\xrightarrow[]{\pi_1}\mathbb{Z}_S,$$ where $\pi_1:I_S\xrightarrow[]{}\mathbb{Z}_S$ is the first projection from the incidence variety
\begin{equation}
\label{eq:IVar}
   I_S= \{(P,\varphi)\in \mathbb{Z}_S\times \pi_k(\mathbb{B}[\![t]\!])\::\:\varphi\in Sol_k(P)\},
\end{equation}
 and $I_S$ can be described  by some integral points ($\overline{\mathbb{Z}}$-rational points) of certain tropical hypersurfaces inside $(\mathbb{T}^*)^{\#S}\times\mathbb{T}$. 

These tropical algebro-geometric properties makes easier to focus on $Sol_k(P)$ rather than $\mu(Sol(P)\!)$ for the case of a general support $S$; in particular, if $Sol_k(P)$ is finite, then $\mu(Sol(P)\!)$ is also finite, so this study can be considered as an extension of \cite{GG26}.

\subsection{Elementary $k$-series and solutions}
\label{sec:ekseries_and_sols}
Elementary $k$-series  arise from a more detailed study of the family of maps \begin{equation*}
    \begin{aligned}
        ord_t(E_M(-)\!):\mathbb{B}[\![t]\!]&\xrightarrow[]{}\overline{\mathbb{N}}\\
        \varphi&\mapsto ord_t(E_M(\varphi)\!)
    \end{aligned}
\end{equation*}
for all $M\in\mathbb{N}^{[k]}\setminus\{0\}$. If $P$ is an \textsc{otde}  of  differential  order $k\in\mathbb{Z}_{>0}$, then we will show that  for any $\varphi\in \mathbb{B}[\![t]\!]$, the tropical sum $ord_t(P(\varphi)\!)\neq\infty$ from \eqref{eq:solution} depends only on an elementary $k$-series $\pi_k(\varphi)\subseteq \varphi$, and we will use this observation to define the set $Sol_k(P)$ of elementary $k$- solutions of $P$.  

We denote by $[k-1]:=\{0,\ldots,k-1\}$ and by $\mathcal{P}([k-1])$ its power set. Since $\mathbb{N}=[k-1]\sqcup [k-1]^c$, any $\varphi\in\mathbb{B}[\![t]\!]$ can be uniquely decomposed as $\varphi=\varphi_{[k-1]}+\varphi_{[k-1]^c}$, where $\varphi_{[k-1]}=\varphi\cap[k-1]$.

\begin{definition}
\label{def:eksN}
    For $k\in\mathbb{Z}_{>0}$ fixed, we define the map $\pi_k:\mathbb{B}[\![t]\!]\xrightarrow[]{}\mathbb{B}[\![t]\!]$ as  
    \begin{equation}
        \pi_k(\varphi)=\begin{cases} \varphi,&\text{if }\varphi\in \mathcal{P}([k-1]),\\
       t^{ord_t\varphi} &\text{if }ord_t(\varphi)\notin[k-1],\\
           
            \varphi_{[k-1]}+t^{ord_t\varphi_{[k-1]^c}},&\text{otherwise.}
        \end{cases}
    \end{equation}
    The set of  elementary $k$-series is the image $Im(\pi_k)\subset \mathbb{B}[\![t]\!]$  of the map $\pi_k$. If $0\neq\varphi\in Im(\pi_k)$, then it is of the form $\varphi=R\cup B$ with
    \begin{enumerate}
         \item $ord_t(\varphi)\notin[k-1]$, then $R=\emptyset$ and  $B=\{q\}$ with $k\leq q<\infty$,
         \item $ord_t(\varphi)\in[k-1]$, then $\emptyset\neq R\subset [k-1]$ and the subcases 
         \begin{enumerate}
             \item $B=\emptyset$,
             \item $B=\{q\}$ with $k\leq q<\infty$,
         \end{enumerate}
    \end{enumerate}
\end{definition}

We now describe some important properties of the map $\pi_k$.

\begin{lemma}
\label{lemma_computations}
    For $\varphi\in\mathbb{B}[\![t]\!]$ and $M\in\mathbb{N}^{[k]}=\bigoplus_{j=0}^k\mathbb{N}\cdot e_j$, we have:
    \begin{enumerate}
        \item $\pi_k(\varphi)\subseteq \varphi$,
        \item $\pi_k(\pi_k(\varphi)\!)=\pi_k(\varphi)$,
        \item $ord_t(\varphi)=ord_t(\pi_k(\varphi)\!)$
        \item $ord_t(E_M(\varphi)\!)= ord_t(E_M(\pi_k(\varphi)\!)$.
    \end{enumerate}
\end{lemma}
\begin{proof}
The first three properties are immediate from Definition \ref{def:eksN}. For the fourth one,  it is enough to show that $ord_t(\frac{d^j\varphi}{dt^j})=ord_t(\frac{d^j\pi_k(\varphi)}{dt^j})$ for all $j\in\{0,\ldots,k\}$. Suppose $\varphi\neq0$ , then we have the following cases: 
   \begin{enumerate}
       \item $q=ord_t(\varphi)\notin[k-1]$. Here $ord_t(\frac{d^j\varphi}{dt^j})=q-j$ for all $j\in\{0,\ldots,k\}$.        On the other hand, we have $\varphi_{[k-1]}=0$ and $B=\{q\}$ with $q\geq k$, so $ord_t(\frac{d^j\pi_k(\varphi)}{dt^j})=q-j$  for all $j\in\{0,\ldots,k\}$.
       \item $ord_t(\varphi)\in[k-1]$. We have two sub-cases: 
       \begin{enumerate}
       \item $B(\varphi)=0$. Here $\varphi=\pi_k(\varphi)$, and we are done.
           \item $B(\varphi)=\{q\}$ with $q\geq k$. If $\emptyset\neq \varphi_{[k-1]}=\{0\leq r_1<\cdots<r_s\leq k-1\}$,  we have  
    \begin{equation*}
        ord_t(\frac{d^j\varphi}{dt^j})=\begin{cases}
            r_i-j,&\text{ if }r_{i-1}<j\leq r_i,\\
            q-j,&\text{ if }r_s+1\leq j\\
        \end{cases}
    \end{equation*}
    and this is clearly just $ord_t(\frac{d^j\pi_k(\varphi)}{dt^j})$, which ends the proof.\qedhere
       \end{enumerate}
   \end{enumerate}
\end{proof}

From now on, if $\emptyset\neq R\in\mathcal{P}([k-1])$, we will denote it as $R=\{0\leq r_1<\cdots<r_s\leq k-1\}$.

Lemma \ref{lemma_computations} represents the basis  for what comes next. In particular, this is enough to prove that the set $\mu(Sol(P)\!)$ is some sort of order-theoretic basis for $Sol(P)$ in the following sense.  
\begin{proposition}
\label{prop:ot_basis}
Let $P=\sum_{M_i\in S}t^{a_{M_i}}E_{M_i}$ be as in \eqref{eq:otde} of order $k$. Then we have\begin{equation*}
         Sol(P)=\bigcup_{\varphi\in \mu(\Sol(P)\!)}\Sol(P)\cap\langle \varphi\rangle^{\uparrow}.
     \end{equation*}
\end{proposition}
\begin{proof}
It is enough to show that any $\varphi\in Sol(P)$ contains a minimal solution. 

By point (4) of Lemma \ref{lemma_computations}, we know that both $\varphi$ and $\pi_k(\varphi)$ yield the same value under the map $ord_t(E_M(-)\!)$ for all $M\in\mathbb{N}^{[k]}$, in particular if $ \varphi\in Sol(P)$, then  $\pi_k(\varphi)\in Sol(P)\cap Im(\pi_k)$ since $P$ has order $k$, and since $\pi_k(\varphi)\subset \varphi$ also by (1) of the same Lemma, it is enough to show that elements in $Sol(P)\cap Im(\pi_k)$ contain a minimal solution, but this is immediate since these elements have finite support.
\end{proof}

\begin{remark}
Lemma \ref{lemma_computations} says that $\pi_k$ is basically an {\it interior operator} on the poset $(\mathcal{P}(\mathbb{N}),\subseteq)$, but not quite, since it does not preserve the order.

In the language of subposets of $(\mathcal{P}(\mathbb{N}),\subset)$, we have shown  that $\mu(Sol(P)\!)\subset Sol(P)\subset\mathcal{P}(\mathbb{N})$ is coinitial.
\end{remark}

\begin{definition}
    If $P$ has order $k$, the set of elementary $k$-solutions of $P$ is $Sol_k(P)=Sol(P)\cap Im(\pi_k)$.
\end{definition}

By Prop.  \ref{prop:ot_basis}, the map $\pi_k:\mathbb{B}[\![t]\!]\xrightarrow[]{}\mathbb{B}[\![t]\!]$ restricts to a map $\pi_k:Sol(P)\xrightarrow[]{}Sol_k(P)$. Going from $Sol(P)$ to $Sol_k(P)$ is very useful, not only because elements of $Sol_k(P)$ all have {\it finite support}, but also due to the inclusion
\begin{equation}
    \mu(Sol(P)\!)\subseteq Sol_k(P).
\end{equation}Indeed, if $\varphi\in \mu(Sol(P)\!)$ is not in $Sol_k(P)$, then $\pi_k(\varphi)\in Sol(P)$ satisfies $\pi_k(\varphi)\subsetneq \varphi$, which contradicts the minimality of $\varphi$. Note that the converse is not true, see Example \ref{ex_Grigoriev}.

In order to go further, we need to compute explicitly the values  $ord_t(E_M(\varphi)\!)\in\overline{\mathbb{N}}$  for all $M\in \mathbb{N}^{[k]}$ and $\varphi\in Im(\pi_k)$.

\begin{definition}
\label{def:values_of_evaluations}
    Given $\sum_{j=0}^km_je_j=M\in \mathbb{N}^{[k]}$ and $R\in\mathcal{P}([k-1])$, we define $|M|_R(z)$ as $|M|_\emptyset(z)=\sum_jm_j(z-j)$, and as 
    \begin{equation}
        |M|_R(z)=\sum_{j=0}^{r_1}m_{j}(r_1-j)+\cdots+\sum_{j=r_{s-1}+1}^{r_s}m_{j}(r_s-j)+\sum_{j=r_s+1}^{k}m_{j}(z-j),
    \end{equation}
    if $\emptyset\neq R=\{0\leq r_1<\cdots<r_s\leq k-1\}$. Here $z$ is a variable.
\end{definition}

 Note that every $|M|_R(z)$ can be expressed as a tropical monomial 
\begin{equation}
\label{eq:def_M_R_(z)}
    |M|_{R}(z)=c(M,R)+d(M,R)z=c(M,R)\odot z^{d(M,R)},\quad c(M,R)\in\mathbb{Z}\text{ and } d(M,R)\in\mathbb{Z}_{\geq0},
\end{equation}
indeed, $c(M,R)=\sum_{j=0}^{r_1}m_{j}(r_1-j)+\cdots+\sum_{j=r_{s-1}+1}^{r_s}m_{j}(r_s-j)-\sum_{j=r_s+1}^{k}jm_{j}$ and $d(M,R)=\sum_{j=r_s+1}^{k}m_{j}$  if $\emptyset\neq R=\{0\leq r_1<\cdots<r_s\leq k-1\}$, while $c(M,\emptyset)=-\sum_jjm_j$ and $d(M,\emptyset)=\sum_jm_j$.  Given a pair  $(\emptyset,\infty)\neq(R,q)\in \mathcal{P}([k-1])\times\overline{\mathbb{N}}_{\geq k}$, we  set  
\begin{equation}
\label{eq:def_M_R_infty}
    |M|_{R}(q):=\begin{cases}
    c(M,R)+d(M,R)q&\text{if }q\neq\infty\\
        c(M,R),&\text{if }d=0\\
        \infty,&\text{if }d>0\text{ and }q=\infty.\\
    \end{cases}
\end{equation}
The next result describes the key property of the expressions $|M|_R(z)$ and their evaluations $|M|_{R}(q)$ on pairs $(\emptyset,\infty)\neq(R,q)\in \mathcal{P}([k-1])\times\overline{\mathbb{N}}_{\geq k}$.
\begin{lemma}
\label{lemma:explicit_computations}
    Let $M\in \mathbb{N}^{[k]}$ and $R\cup B=\varphi\in Im(\pi_k)$. Then $ord_t(E_M(\varphi)\!)=|M|_R(ord_t(B)\!)$.
\end{lemma}
\begin{proof}
    If $R\cup B=\varphi\in Im(\pi_k)$, it defines the pair $(R,q)=(R,ord_t(B)\!)\in  \mathcal{P}([k-1])\times\overline{\mathbb{N}}_{\geq k}$. We now check the three different cases from Def. \ref{def:eksN}:
    \begin{enumerate}
        \item if $ord_t(\varphi)\notin[k-1]$, then $R=\emptyset$ and  $B=\{q\}$ with $k\leq q<\infty$. Thus $ord_t(E_M(\varphi)\!)=\sum_jm_j(q-j)=|M|_\emptyset(ord_t(B)\!)$, since $q=ord_t(B)$.
        \item if $ord_t(\varphi)\in[k-1]$, then $\emptyset\neq R=\{0\leq r_1<\cdots<r_s\leq k-1\}$. For $M=\sum_j^km_je_j$, we define its support $Supp(M)=\{i\::\:m_i\neq0\}\subseteq[0,k]$. Now $ord_t(E_M(\varphi)\!)$ equals

\begin{equation*}
    \begin{cases}
    \sum_{j=0}^{r_1}m_{j}(r_1-j)+\cdots+\sum_{j=r_{s-1}+1}^{r_s}m_{j}(r_s-j),& Supp(M)\subset [0,r_s], \\
    \sum_{j=0}^{r_1}m_{j}(r_1-j)+\cdots+\sum_{j=r_{s-1}+1}^{r_s}m_{j}(r_s-j)+\sum_{j=r_s+1}^{k}m_{j}(q-j),&Supp(M)\not\subset [0,r_s],B=\{q\}\\
    \infty,&Supp(M)\not\subset [0,r_s], B=0.\\
    \end{cases}
\end{equation*}
    \end{enumerate}

    By  \eqref{eq:def_M_R_(z)}, we have $|M|_{R}(z)=c(M,R)+d(M,R)z$  for some $c(M,R)\in\mathbb{Z}$ and $\sum_{j=r_s+1}^{k}m_{j}=d(M,R)\in\mathbb{Z}_{\geq0}$. If $B=\{q\}\neq0$, then $|M|_R(ord_t(B)\!)=ord_t(E_M(\varphi)\!)$, and if $B=0$, this is, $q=\infty$, then $ord_t(E_M(\varphi)\!)=|M|_R(ord_t(B)\!)$ follows from the observation that  $d(M,R)=0$ if and only if $Supp(M)\subset [0,r_s]$.
\end{proof}

A first question may be to check the structure of the set   $Sol_k(P)$ for a non-linear \textsc{otde} $P$, and to find conditions upon which $Sol_k(P)$ is finite. Since $\mu(Sol(P)\!)\subset Sol_k(P)$, this study also yields results on the set of minimal solutions of $P$; for this purpose we will construct some tropical polynomials in the next section, using the objects $|M|_{R}(z)$ and their evaluations $|M|_{R}(q)$.
\subsection{Discriminants on spaces of parameters}
\label{sect:discriminants}
Following Convention \ref{conv:shape_of_P}, in this section, we consider $P=\sum_{M_i\in S}t^{a_{M_i}}E_{M_i}$ as in \eqref{eq:otde} with $ord(P)=k\in\mathbb{Z}_{>0}$, and $\#Supp(P)=s>1$.  In order to study the  set of elementary $k$-solutions of $P$ from an algebraic perspective, we will abstract these conditions into the following definition.
\begin{definition}
\label{def:supsetk}
    A support set of order $k\in\mathbb{Z}_{>0}$ is a finite subset $\{M_1,\ldots,M_s\}=S\subset\mathbb{N}^{[k]}$ satisfying the following conditions:
\begin{enumerate}
\item $0\notin S$ ($P$ has no constant term)
    \item $s:=\#S>1$ ($P$ is not a monomial), and 
    \item there exists  $\sum_im_ie_i=M\in S$ with $m_k\neq0$ ($P$ is a differential equation of order $k>0$).
\end{enumerate}
\end{definition}

\begin{convention}
    From now on, $k\in\mathbb{Z}_{>0}$ will be fixed, and $S=\{M_1,\ldots,M_s\}\subset\mathbb{N}^{k}$ will denote a support set (of differential order $k$). Also, if $0\neq R\cup B\in\pi_k(\mathbb{B}[\![t]\!])$, we will sometimes identify it with the pair $(\emptyset,\infty)\neq (R,\ord_t(B)\!)\in\mathcal{P}([k-1])\times\overline{\mathbb{N}}_{\geq k}$, as in Lemma \ref{lemma:explicit_computations}.
\end{convention}


If $\sum_{i=0}^{k}m_{i}e_i=M\in \mathbb{N}^{[k]}$, we define $ord(M)=\max\{i\::\:m_i\neq0\}$ and the degree of $M$ is $deg(M)=\sum_im_{i}$.

The torus $\mathbb{Z}^s\cong (\overline{\mathbb{Z}}^*)^s$ equipped with coordinates $\{x_M\::\:M\in S\}=\{x_{M_1},\ldots,x_{M_s}\}$ is the space of parameters of equations \eqref{eq:otde} having support $S$ by identifying $P=\sum_{M\in S} t^{a_M}E_M$ with its vector of coefficients $c(P):=(a_M\::\:M\in S)\in \mathbb{Z}^s$; if this is the case, then  $ord(P)=ord(S)=k$. In order to connect with tropical geometry, we will use the tropical torus $(\mathbb{T}^*)^{s}$ equipped with coordinates $\{x_M\::\:M\in S\}=\{x_{M_1},\ldots,x_{M_s}\}$ instead, and identify  $(\overline{\mathbb{Z}}^*)^s$ with the set  $\mathbb{Z}_S= \mathbb{Z}^{s}\subset (\mathbb{T}^*)^{s}$ of $\mathbb{Z}$-rational points.  We just write $(a_M)=P\in \mathbb{Z}_S$ from now on.

We are now ready to introduce the family of discriminants. We will use the tropical notation. 
\begin{definition}
\label{def:disc_pols}
    For $S$ a support set and $R\in\mathcal{P}([k-1])$, we define 
    \begin{equation*}
        \Delta_{S,R}:=\bigoplus_{M\in S}|M|_R(z)\odot x_M\in\mathbb{T}[x_{M_1},\ldots,x_{M_s},z],
    \end{equation*}
    where  the terms  $|M|_R(z)$ were introduced in Definition \ref{def:values_of_evaluations}.
\end{definition}

We consider the space $(\mathbb{T}^*)^{\#S}\times \mathbb{T}$ endowed with coordinates $\{x_{M_1},\ldots,x_{M_s},z\}$. For each $\Delta_{S,R}$, we get an evaluation function 
\begin{equation*}
    \begin{aligned}
        ev_{\Delta_{S,R}}:(\mathbb{T}^*)^{\#S}\times \mathbb{T}&\xrightarrow[]{}\mathbb{T}\\
        (P=(a_M),q)&\mapsto \min_{M\in S}\{|M|_R(q)+ a_M\},
    \end{aligned}
\end{equation*}
where the evaluations $|M|_R(q)$ come from \eqref{eq:def_M_R_infty},  and  we denote by $V(\Delta_{S,R})$ its tropical hypersurface. 

\begin{proposition}
\label{prop:discs}
    Let $S$ be a support set, and let $R\cup B=\varphi\in Im(\pi_k)$. Then $\varphi\in Sol(P)$ for $P\in\mathbb{Z}_S$ if and only if $(P,ord_t(B)\!)\in V(\Delta_{S,R})$.
\end{proposition}
\begin{proof}
We have $ev_{\Delta_{S,R}}(P,ord_t(B)\!)=\min_{M\in S}\{|M|_R(ord_t(B)\!)+ a_M\}$, and from Lemma \ref{lemma:explicit_computations}, we have that  $ord_t(E_M(\varphi)\!)=|M|_R(ord_t(B)\!)$ for every $M\in \mathbb{N}^{[k]}$ and every  elementary $k$-solution $\varphi=R\cup B$, so $ev_{\Delta_{S,R}}(P,ord_t(B)\!)=\min_{M\in S}\{ord_t(E_M(\varphi)\!)+ a_M\}=ord_t(P(\varphi)\!)$, and one vanishes tropically if and only if the other does.
\end{proof}

The relevance of the set $Sol_k(P)$ over the set $\mu(Sol(P)\!)$ stems from the fact that, once we fix a support set $S$ of order $k$, the family  $\{Sol_k(P)\::\:P\in \mathbb{Z}_S\}$ can be organized  in a bundle $Sol_k(P)\xrightarrow[]{}I_S\xrightarrow[]{\pi_1}\mathbb{Z}_S$, where by Prop. \ref{prop:discs}, the incidence variety $I_S= \{(P,\varphi)\in \mathbb{Z}_S\times \pi_k(\mathbb{B}[\![t]\!])\::\:\varphi\in Sol_k(P)\}$ from \eqref{eq:IVar} admits the following description:

\begin{equation*}
     I_S=(V(\Delta_{S,\emptyset})\cap(\mathbb{Z}^s\times \mathbb{Z}_{\geq k})\!)\cup\bigcup_{\emptyset\neq R\subseteq[k-1]}V(\Delta_{S,R})\cap(\mathbb{Z}^s\times\overline{\mathbb{Z}}_{\geq k})\subset V(\bigodot_{R\in\mathcal{P}(k-1)}\Delta_{S,R})\cap(\mathbb{Z}^s\times\overline{\mathbb{Z}}_{\geq k}).
 \end{equation*}

 So, the family the family  $\{Sol_k(P)\::\:P\in \mathbb{Z}_S\}$ is tropical semi-algebraic, and it can be studied from a geometric point of view through the diagram $$\mathbb{T}\xrightarrow[]{}V(\bigodot_{R\in\mathcal{P}(k-1)}\Delta_{S,R})\xrightarrow[]{\pi_1}(\mathbb{T}^*)^s.$$

 This is the key property behind the family tropical hypersurfaces $\{V(\Delta_{S,R})\subset (\mathbb{T}^*)^{\#S}\times \mathbb{T}\::\:R\in\mathcal{P}([k-1])\}$, and in the coming sections we will study their properties for a fixed support set $S$ of differential order $k$. Since it is cumbersome keeping track of the integer points of the tropical hypersurfaces $V(\Delta_{S,R})\subset (\mathbb{T}^*)^{\#S}\times \mathbb{T}$, sometimes we will not indicate the intersections.

 \begin{remark}
     The family $\{\Delta_{S,R}\::\:R\in\mathcal{P}([k-1])\}$ for a fixed support set $S$ is a direct generalization of the family of polynomials $\{A_j=min_{0\leq i\leq min\{j,k\}}\{x_i+(j-i)\}\}_{j=0,\ldots,k}$ introduced in \cite[Definition 5]{GG26} for the linear case (see our conventions in Example \ref{ex:lin_case}). See Example \ref{ex:A_k}.
\end{remark}

\subsection{Minimal solutions $ord(\varphi)\notin\{[ord(P)-1]\}$}
\label{Sec:ordgeqk}
In \cite{GG26}, the set $\mathcal{C}_{\mathbb{Z}_{\geq0}}(P)$  of minimal solutions of $P$ having $t$-adic order at least $ord(P)$ was introduced. Here we do the same for $Sol_k(P,[k-1]^c)$, but using the dichotomy introduced in Lemma \ref{lemma_computations},  if $ord_t(\varphi)$ falls on $[k-1]=\{0,\ldots,k-1\}$ or on $[k-1]^c$; this will be useful in the coming sections.

\begin{definition}
   We denote by $Sol_k(P,[k-1]^c)$ (respectively by $\mu(P,[k-1]^c)$) the set  of elementary $k$-solutions of $P$ (respectively the set  of minimal solutions of $P$) having $t$-adic order away from $[k-1]$.
\end{definition}

In this section, we fix a support set $S$ and study the  set $\mu(P,[k-1]^c)$ as $P$ varies in $\mathbb{Z}_S$. In order to do so,  we use the polynomial from Def. \ref{def:disc_pols} associated to elementary $k$-series $\varphi=R\cup B$ of shape $R=\emptyset$ and $B=\{q\}$, this is $\Delta_{S,\emptyset}(x_{M_1},\ldots,x_{M_s},z):=\bigoplus_j|M_j|_\emptyset(z)\odot x_{M_j}$, where 
$$|M|_\emptyset(z)=\sum_jm_j(z-j)=\sum_jm_j(z-k)+\sum_jm_j(k-j),$$
by Definition \ref{def:values_of_evaluations}. It will be convenient to use the following alternative presentation \eqref{eq:polyinw} below for the evaluation function $ev_{\Delta_{S,\emptyset}}:(\mathbb{T}^*)^{\#S}\times \mathbb{T}^*\xrightarrow[]{}\mathbb{T}$ induced by $\Delta_{S,\emptyset}$ (note that $q\neq\infty$), using the change of variables $w=z-k$
\begin{equation}
\label{eq:polyinw}
    \Delta_{S,\emptyset}(x_{M_1},\ldots,x_{M_s},w):=min_j\{|M_j|_k+deg(M_j)w+x_{M_j}\},
\end{equation}
where for $\sum_{i=0}^{k}m_{j,i}e_i=M_j\in S$,  we write  $deg (M_j)=\sum_{i=0}^{k}m_{j,i}$, and  $|M_j|_k:=\sum_{i=0}^{k}(k-i)m_{j,i}$.   
\begin{proposition}
\label{prop:poly}
Let $P\in\mathbb{Z}_S$. If $\varphi\in \mu(P,[k-1]^c)$, then $\varphi=t^q$ for some $q\geq k$. Moreover, $t^q\in \mu(P,[k-1]^c)$  if and only if $(P,q-k)\in V(\Delta_{S,\emptyset}(x_{M_1},\ldots,x_{M_s},w)\!)$.
\end{proposition}
\begin{proof}
The inclusion $\mu(Sol(P)\!)\subset Sol_k(P)$ induces  $\mu(P,[k-1]^c)\subset Sol_k(P,[k-1]^c)$, and it follows from Def. \ref{def:eksN} that any element $\varphi\in Sol_k(P,[k-1]^c)$ is a monomial, so it is minimal, and we get $\mu(P,[k-1]^c)= Sol_k(P,[k-1]^c)$.

Let $t^q$ with $q\geq k$, then it follows from Prop. \ref{prop:discs} that $t^q\in Sol_k(P)$ if and only if $(P,q)\in V(\Delta_{S,\emptyset}(x_{M_1},\ldots,x_{M_s},z)\!)$, but this is equivalent to  $(P,q-k)\in V(\Delta_{S,\emptyset}(x_{M_1},\ldots,x_{M_s},w)\!)$ since $w=z-k$. Finally, note that , which ends the proof.
\end{proof}

We have an identification of the incidence variety 
\begin{equation}
\label{eq:incidence_1}
    \{(P,q)\in\mathbb{Z}_S\times\mathbb{Z}_{\geq k}\::\:q\in\mu(P,[k-1]^c)\}\longleftrightarrow V(\Delta_{S,\emptyset})\cap \mathbb{Z}^s\times \mathbb{Z}_{\geq0}.
\end{equation}

\begin{definition}
\label{def:degrees_of_a_support_set}
 We denote by   $deg(S)=\{deg(M)\::\:M\in S\}=\{d_1<\cdots<d_m\}$.  If $deg(S)=\{d\}$, then we say that $S$ is homogeneous of degree $d$.
\end{definition}

If $S$ is homogeneous, the specialization 
\begin{equation}
\label{eq:disc_homogeneous}
    \Delta_{S}:=\Delta_{S,\emptyset}(x_{M_1},\ldots,x_{M_s},0)=min_j\{\sum_{i=0}^{k}(k-i)m_{j,i}+x_{M_j}\}
\end{equation}
controls the cardinality of the set $\mu(P,[k-1]^c)$ in the  sense of the following result.
\begin{corollary}
\label{prop:hcase}
    Let $S$ be homogeneous. If $P\in\mathbb{Z}_S$, then $ \mu(P,[k-1]^c)$ satisfies \begin{equation*}
    \mu(P,[k-1]^c)=\begin{cases}
        \emptyset, &\text{ if }P\notin V(\Delta_S),\\
        \{t^{k+i}\::\:i\in \mathbb{Z}_{\geq0}\},&\text{ otherwise}.\\
    \end{cases}
\end{equation*} 
\end{corollary}

\begin{proof}
Let $S=\{M_1,\ldots,M_s\}$ be homogeneous of degree $d$.
By Prop. \ref{prop:poly}, we have that if $t^q\in \mu(P,[k-1]^c)$ and  $r=q-k\geq0$, then 
\begin{equation*}
        ev_{\Delta_{S,\emptyset}}(c(P),r)=min_j\{|M_j|_k+dr+a_j\}=ev_{\Delta_{S,\emptyset}}(c(P),0)+dr,
    \end{equation*}
    
and  $ev_{\Delta_{S,\emptyset}}(c(P),r)$ vanishes tropically if and only if $ev_{\Delta_{S,\emptyset}}(c(P),0)$ vanishes tropically.
\end{proof}

\begin{example}
\label{ex:A_k}
    In the  case or a linear differential polynomial $P$ (see Example \ref{ex:lin_case}), we have 
    \begin{equation*}
        \Delta_{S,\emptyset}(x_{e_0},\ldots,x_{e_k},w)=min_j\{(k-j)+w+x_{e_j}\}
    \end{equation*} 
    and $\Delta_{S,\emptyset}(x_{e_0},\ldots,x_{e_k},w)|_{w=0}=min_j\{(k-j)+x_{e_j}\}$  is precisely the polynomial $A_k$ appearing in  \cite[Theorem 1]{GG26}.
\end{example}

What happens with $\mu(P,[k-1]^c)$ if $P\in\mathbb{Z}_S$ and $S$ is not homogeneous? For every $P=(a_M)\in\mathbb{Z}_S$, we can consider the fiber $\pi^{-1}(P)\cong\mathbb{T}^*$ of the first projection $\pi_1:(\mathbb{T}^*)^s\times\mathbb{T}^*\xrightarrow[]{}(\mathbb{T}^*)^s$, so that $\mu(P,[k-1]^c)$ correspond to the $\mathbb{Z}_{\geq0}$ points in $\pi^{-1}(P)\cap V(\Delta_{S,\emptyset})\subset \mathbb{T}^*$. Consider now the   specialization \eqref{eq:disc:specialization} of \eqref{eq:disc}:
\begin{equation*}
    \Delta_{S,\emptyset}|_P(w)=\Delta_{S,\emptyset}(a_{M_1},\ldots,a_{M_s},w):=min_j\{|M_j|_k+a_{M_j}+deg(M_j)w\}
\end{equation*}

To keep notation simple, we will denote $\Delta_S|_P(w)$ the evaluation function $\mathbb{T}^*\xrightarrow[]{}\mathbb{T}$ induced by $\Delta_{S,\emptyset}|_P(w)$, instead of $ev_{\Delta_{S,\emptyset}|_P(w)}$. This  evaluation function  in general {\it is not a tropical polynomial}, since there may be many monomials $|M_j|_k+a_{M_j}+deg(M_j)w$ with the same degree in $w$ (See Remark \ref{rem:coro}). For this reason,  we will keep this notation as an expanded tropical sum. Nevertheless,  we use our conventions for tropical polynomials, namely,  we will denote by $V(\Delta_S|_P(w)\!)$ its {\it tropical hypersurface}, and if $q\in V(\Delta_S|_P(w)\!)$, then $S(q)=\{M\in S\::\:\Delta_S|_P(q)=|M|_S+a_{M}+deg(M)q\}$.  See Definition \ref{def:set_of_exponents}. We will discuss the properties of these objects in Section \ref{sec:ordatleastk}.

Suppose that $S$ is not homogeneous and let $S_{0}\subset S$ be the subset consisting of monomials of minimal degree. Let $\emptyset\neq S_1=S\setminus S_0$, and write $\mathbb{T}^{\#S}\cong \mathbb{T}^{\#S_0}\times \mathbb{T}^{\#S_1}$. If $a\in \mathbb{T}^{\#S}$, then we write as $a=(a_0,a_1)$ with $a_i\in \mathbb{T}^{\#S_i}$. In particular, if $P\in\mathbb{Z}_S\subset \mathbb{T}^{\#S}$, we write $P=P_0+P_1$ with $P_i\in\mathbb{Z}_{S_i}\subset \mathbb{T}^{\#S_i}$.

\begin{proposition}
\label{prop:nothcase}
    Let $S$ be not homogeneous and let $P_0+P_1=P\in\mathbb{Z}_S$. Then     $\#\mu(P,[k-1]^c)=\infty$ if and only if $        P_0\in V(\Delta_{S_0}(x_M\::\:M\in S_0)\!).$
\end{proposition}

\begin{proof}
    Under our previous conventions,  Prop. \ref{prop:poly} says that for  $P\in\mathbb{Z}_S$ we have  a bijection of sets 
    \begin{equation}
        \begin{aligned}
            \mu(P,[k-1]^c)&\xrightarrow[]{\sim} V(\Delta_S|_P(w)\!)\cap\mathbb{Z}_{\geq0}\\
            q&\mapsto q-k
        \end{aligned}
    \end{equation}
        so we need to analyze the monomials $|M_j|_k+a_{M_j}+deg(M_j)w$ from $ \Delta_S|_P(w)$ for $w\geq0$. These are Euclidean rays $w\mapsto C+dw$ for $C\in\mathbb{R}$ and $d\in deg(S)=\{d_1<\cdots<d_m\}$ a positive integer.

    For $w>\!\!>0$, the value $\Delta_S|_P(w)$ will be taken at a monomial $|M_j|_k+a_{M_j}+deg(M_j)w$ with $deg(M_j)\in S_0$, and the only way in which $\Delta_S|_P(w)$ could vanish tropically in a ray $[q,\infty)$ is if there exists another monomial $|M_b|_k+a_{M_b}+deg(M_b)w$ defining the same Euclidean ray, in particular $deg(M_b)\in S_0$. Since $S_0$ is homogeneous, by  Cor. \ref{prop:hcase} this happens if and only if $P_0\in V(\Delta_{S_0}(x_M\::\:M\in S_0)\!)$ as in \eqref{eq:disc_homogeneous},  which finishes the proof.
\end{proof}

\begin{remark}
\label{rem:coro}
    Note that $S$ is homogeneous of degree $d$, then $S_{min }= S=\{d\}$. In this case the specialization \eqref{eq:disc:specialization} takes the form 
\begin{equation*}
\Delta_{S,\emptyset}|_P(w)=min_j\{|M_j|_S+a_{M_j}+dw\}=min_j\{|M_j|_S+a_{M_j}\}+dw.
\end{equation*}
Thus, we can recover Cor. \ref{prop:hcase} from Prop. \ref{prop:nothcase}. Note that if $C:=min_j\{|M_j|_S+a_{M_j}\}$, then the {\it polynomial} form of $\Delta_{S,\emptyset}|_P(w)=C+dw=C\odot w^{\odot d}$ would be just a monomial (in particular, it's bend locus in $\mathbb{T}^*$ would be empty).
\end{remark}

If $S$ is not homogeneous, since $Supp(\Delta_{S_0,\emptyset})=S_0$, we can compute $V(\Delta_{S_0,\emptyset})\subset (\mathbb{T}^*)^{\#S_0}$, and consider $V(\Delta_{S_0,\emptyset})\times \mathbb{T}^{\#(S\setminus S_0)}\subset \mathbb{T}^s$. Note that $V(\Delta_{S_0,\emptyset})=\{\infty\}$ if and only if $\#S_{0}=1$.

In the coming sections we will be more concerned about the particular properties of $Sol_k(P)$ for $P\in \mathbb{Z}_S\setminus V(\Delta_{S_0}(x_M\::\:M\in S_0)\!)$. In the decomposition $\varphi=R\cup B$,  the set $Sol_k(P,[k-1])$ corresponds  case $R\neq\emptyset$ will be studied in Section \ref{sect:F(P)}. The case $R=\emptyset$ will be studied in Section \ref{sec:ordatleastk}, after the introduction of our definition of  (minimal) tropical Laurent (and Hahn) series solution of \textsc{otde}s.

\section{Elementary $k$-solutions $ord(\varphi)\in\{[ord(P)-1]\}$}
\label{sect:F(P)}
We denote by $Sol_k(P,[k-1])$ (respectively $\mu(P,[k-1])$) the set elementary $k$-solutions $\varphi$ of $P$ (respectively minimal) satisfying $ord_t(\varphi)\in[k-1]$. In terms of the shape $\varphi=R\cup B$, this is the case $\emptyset\neq R=\{0\leq r_1<\cdots<r_s\leq k-1\}$, which has two sub-cases :
     \begin{enumerate}
         \item $B=\emptyset$, then $\varphi=R\subset[k-1]$. We will discuss this case in Section \ref{sec:ordatmostkmo}, related to genericity conditions.
         \item $B\neq\emptyset$, then $\varphi=R\cup q$.  We will discuss this case in Section \ref{sec.generic_solutions}.
     \end{enumerate} 

The set $\mu(P,[k-1])$ was denoted by $F(P)$ in  \cite{GG26}, and we have $\mu(P,[k-1])\subset Sol_k(P,[k-1])$. Recall that if $S$ is a support set and the space $(\mathbb{T}^*)^{\#S}\times \mathbb{T}$ is endowed with coordinates $\{x_{M_1},\ldots,x_{M_s},z\}$, for $\emptyset\neq R\in\mathcal{P}([k-1])$ we can consider the (incidence) variety $V(\Delta_{S,R})\subset (\mathbb{T}^*)^{\#S}\times \mathbb{T}$, and we have
\begin{equation*}
V(\Delta_{S,R})\cap(\mathbb{Z}^s\times\overline{\mathbb{Z}}_{\geq k})=\{(P,ord(B)\!)\in\mathbb{Z}_S\times\overline{\mathbb{Z}}_{\geq k}\::\:R\cup B\in Sol_k(P)\}.
\end{equation*}

\subsection{Solutions $R\subset[k-1]$ and regularity}
\label{sec:ordatmostkmo}
In this section we will analyze elementary $k$-solutions $\varphi=R\cup B$   with $\emptyset\neq  R\subset[k-1]$ and $B=0$. For this we fix $S=\{M_1,\ldots,M_s\}$ of order $k$ and we use the family of polynomials 
\begin{equation*}
\Delta_{S,R}(x_{M_1},\ldots,x_{M_s},z)|_{z=\infty}:=min_j\{|M_j|_R(\infty)+x_{M_j}\},\quad \emptyset\neq  R\subset[k-1],
    \end{equation*} 
where $|M_j|_R(\infty)$ comes from \eqref{eq:def_M_R_infty}. Since $q=\infty$ is constant, each $\Delta_{S,R}$ defines an evaluation function $ev_{\Delta_{S,R}}:(\mathbb{T}^*)^{\#S}\xrightarrow[]{}\mathbb{T}$. Recall that a tropical prevariety is a finite intersection of tropical hypersurfaces, they are always closed sets in $\mathbb{T}^n$ (with respect to the Euclidean topology). For every $\emptyset\neq R\subset[k-1]$ we set $Y_R=V(\Delta_{S,R}(x_{M_1},\ldots,x_{M_s},z)|_{z=\infty})$  the bend locus of the  evaluation map $ev_{\Delta_{S,R}}:\mathbb{T}^{s}\xrightarrow[]{}\mathbb{T}$, and 
\begin{equation}
\label{eq:deg_sols}
    \bigcap_{\emptyset\neq R\subset[k-1]}Y_R:=Y(S)\subset(\mathbb{T}^*)^{\#S}.
\end{equation}

The main result of this section is the following 
\begin{proposition}
For $S$ as above, we have that $P\in \mathbb{Z}_S$ has no solution contained in $[k-1]$ if and only if $P\notin Y(S)$.
\end{proposition}
\begin{proof}
Let us consider $\emptyset\neq R\subset[k-1]$.  It follows from Prop. \ref{prop:discs} that  $R\in Sol(P)$ if and only if $\Delta_{S,R}(x_{M_1},\ldots,x_{M_s},\infty)$ vanishes at $P$, and  since $Y_R=V(\Delta_{S,R})$, then $\mathbb{Z}_S\cap V(\Delta_{S,R})=\{P\in \mathbb{Z}_S\::\:R\in Sol(P)\}$ and the result follows.
\end{proof}

\begin{remark}
\label{prop:ud}
Another way to state the previous result is the following:  $P\in Y_R$ if and only if $\varphi=R\cup\{q\}\in Sol_k(P)$ for all $q>>0$. Indeed, if $P=(a_M)\in Y_R$, then there are two monomials of $\Delta_{S,R}(a_{M_1},\ldots,a_{M_s},z)|_{z=\infty}$ where the minimum is achieved (and it is different from $\infty$). There exists $q_0>>0$ such that when we evaluate the monomials $M$ satisfying $Supp(M)\not\subset [0,r_s]$ at $q_0$, then this evaluation is still bigger than the value $\Delta_{S,R}(a_{M_1},\ldots,a_{M_s},z)|_{z=\infty}$, and this condition is still satisfied for $q>q_0$.

Conversely, suppose that $\varphi=R\cup\{q\}\in Sol_k(P)$ for $q>>0$, but $P\notin Y_R$, then there exists $M_0\in S$ such that $\infty\neq ord_t(P(R)\!)=a_{M_0}+ord_t(E_{M_0}(R)\!)<a_M+ord_t(E_M(\varphi)\!)$ for all $M\neq M_0$, it follows that $Supp(M_0)\subset [0,r_s]$ by Lemma \ref{lemma:explicit_computations}. On the other hand, by hypothesis we have that there exists at least two different monomials $M_1,M_2\in S$ such that $\infty\neq ord_t(P(R\cup\{q\})\!)=a_{M_i}+ord_t(E_{M_i}(R)\!)\leq a_M+ord_t(E_M(\varphi)\!)$  for all $M\in S$, in particular $ord_t(P(R\cup\{q\})\!)\leq ord_t(P(R)\!)=a_{M_0}+ord_t(E_{M_0}(R)\!)$ for $q>>0$, which is not possible.
    
This means that there exist solutions $\emptyset \neq R\subset[k-1]$ for which there exists $k\leq q<\infty$ forming a new {\it artificial} solution $\varphi=R\cup\{q\}$, which of course can't be minimal. 
\end{remark}

Let $X(S)\subset\mathbb{T}^s$ be the tropical prevariety $X(S):=V(\Delta_{S_0,\emptyset},\Delta_{S,R},\emptyset\neq R\subset[k-1])=Y(S)\cap V(\Delta_{S_0}(x_M\::\:M\in S_0)\!)$, where $S_{0}\subset S$ is the subset consisting of monomials of minimal degree. We now extend the  concept of regular polynomial from \cite[Definition 8]{GG26}.
\begin{definition}
\label{def:ext_regular}
    We say that $P=(a_M)\in \mathbb{Z}_S$ is regular if $P\notin X(S)$.
\end{definition}

We now extend  \cite[Proposition 5]{GG26}.

\begin{corollary}
\label{coro:coro1}
We have that $P$ is regular if and only if $\mu(Sol(P)\!)=\mu(P,[k-1])\sqcup\mu(P,[k-1]^c)$ satisfies: 
\begin{enumerate}
        \item any $\varphi\in \mu(P,[k-1])$ is of the form $R\cup\{q\}$, with $\emptyset\neq R\subset [k-1]$ and $q\in\mathbb{Z}_{\geq k}$ varying in a finite set,
        \item $\mu(P,[k-1]^c)\cong V(\Delta_{S,\emptyset}|_P(w)\!)\cap\mathbb{Z}_{\geq0}$ is finite (empty, if $S$ is homogeneous).
    \end{enumerate}
    Here $\Delta_{S,\emptyset}|_P(w)=min_j\{|M_j|_S+a_{M_j}+deg(M_j)w\}$. 
\end{corollary}
\begin{proof}
    We only lack the statement (1), which will be proved below in Theorem \ref{thm:crit_fin}. 
\end{proof}

We will enhance later Corollary \ref{coro:coro1}  into Corollary \ref{cor:2}, to consider also minimal tropical Laurent solutions.

Note that the support of $\Delta_{S,R}$ is $\{M\in S,\:Supp(M)\subset [0,r_s]\}$, so not all the polynomials are relevant in the computation of the tropical hypersurface $Y(S)$.

In the linear case $S=\{e_0,\ldots,e_k\}$, the set $\mathbb{Z}_S\setminus \mathbb{Z}_S\cap X(S)$ was  called the set of regular tropical differential polynomials (see \cite[Definition 8]{GG26}). For this case, regularity is a combination of  {\it tropical holonomy} (finite number of minimal solutions) and {\it genericity} of tropical minimal solutions, which should be of the form $R\cup\{q\}$, with $\emptyset\neq R\subset [k-1]$ and $q=q(R)\geq k$ {\it unique}. 

For this  case, elementary $k$-solutions of the form $R\cup B$ with both $R\neq\emptyset$ and $B\neq0$ are the only ones that remain to be studied. We do this in the following section.

\subsection{Minimal solutions $R\cup B$ with both $R\neq\emptyset$ and $B\neq0$}
\label{sec.generic_solutions}

In this section, we study the  subset  of $Sol_k(P,[k-1])$ consisting of elementary $k$-solutions $\varphi=R\cup B$ with $R\neq\emptyset$ and $B\neq 0$. For this we fix $S=\{M_1,\ldots,M_s\}$ of order $k$ and we use the family of polynomials 
\begin{equation}
\label{eq:disc_T}
\Delta_{S,R}(x_{M_1},\ldots,x_{M_s},z):=min_j\{|M_j|_R(z)+x_{M_j}\},\quad \emptyset\neq  R\subset[k-1],
    \end{equation} 
where $|M_j|_R(z)$ comes from Definition \ref{def:values_of_evaluations}. Again, we shall consider the evaluation function $ev_{\Delta_S}:(\mathbb{T}^*)^{\#S}\times \mathbb{T}\xrightarrow[]{}\mathbb{T}$, where  the space $(\mathbb{T}^*)^{\#S}\times \mathbb{T}$ endowed with coordinates $\{x_{M_1},\ldots,x_{M_s},z\}$.

We have the main result of this Section. Recall \eqref{eq:sing} that \begin{equation*}
Sing(S)=    \{P\in\mathbb{Z}_S\::\:\#Sol_k(P)=\infty\}.
\end{equation*}

\begin{theorem}\label{thm:crit_fin}
     If $P\in\mathbb{Z}_S$, then  $\#Sol_k(P)<\infty$ if and only if $P$ is regular.    
\end{theorem}

\begin{proof}
Let $S_0\subset S$ be the subset consisting of monomials of minimal degree. Recall that $P$ is regular if and only if $P\notin X(S)$, where $X(S)=Y(S)\cap V(\Delta_{S_0}(x_M\::\:M\in S_0)\!)$.

We have that $Sol_k(P)$ for $P\in\mathbb{Z}_S$ is finite if and only if both $Sol_k(P,[k-1]^c)$ and $Sol_k(P,[k-1])$ are finite, since $Sol_k(P)=Sol_k(P,[k-1])\sqcup Sol_k(P,[k-1]^c)$. 

The first condition $\#Sol_k(P,[k-1]^c)<\infty$ if and only if $P\notin V(\Delta_{S_0})$ for the  case of non-homogeneous  $S$ comes from Proposition \ref{prop:nothcase}, and it comes from Cor. \ref{prop:hcase}  for $S$ homogeneous.

For the second condition $\#Sol_k(P,[k-1])$, If we take the first projection $\pi_1:(\mathbb{T}^*)^{\#S}\times \mathbb{T}\xrightarrow[]{}(\mathbb{T}^*)^{\#S}$, and $P\in \mathbb{Z}_S$, then $\pi_1^{-1}(P)=\mathbb{T}$ is a tropical linear variety of dimension 1, then we have that $\{q\in\mathbb{Z}_{\geq k}\::\:R\cup\{q\}\in Sol_k(P)\}=V(\Delta_{S,R})\cap \mathbb{Z}_{\geq k}\subset\pi_1^{-1}(P)$, and we have from Rem. \ref{prop:ud} that $\#\{q\in\mathbb{Z}_{\geq k}\::\:R\cup\{q\}\in Sol_k(P)\}<\infty$ if and only if $P\notin Y_R$, so $Sol_k(P,[k-1])$ is finite  if and only if $P\notin Y(S)$. \end{proof}

Now assume that $\emptyset\neq R=\{0\leq r_1<\cdots<r_s\leq k-1\}$. If $P=(a_{M_j})\in (\mathbb{T}^*)^{\#S}$, then we have a specialization for $\Delta_{S,R}$ as \eqref{eq:disc:specialization} for the case $R=\emptyset$:

\begin{equation*}
    \Delta_{S,R}|_P(z)=\bigoplus_jc(M_j,R)\odot a_{M_j}\odot z^{d(M_j,R)},
\end{equation*}

and a similar identification for $P\in\mathbb{Z}_S$:
\begin{equation}
\{R\cup q\in Sol_k(P)\}\xrightarrow[]{\sim} V(\Delta_{S,R}|_P(z)\!)\cap\mathbb{Z}_{\geq k}
    \end{equation}

Let $q\in V(\Delta_{S,R}|_P(z)\!)\cap\mathbb{Z}_{\geq k}$. Let $\varphi=R\cup q$, then we claim that there is $M_j\in S(\varphi)$ such that $Supp(M_j)\not\subset [0,r_s]$, where $ S(\varphi)$ is the set from Def. \ref{def:set_of_exponents}.

If $Supp(M_i)\subset [0,r_s]$ for all $M_i\in S(\varphi)$, then $ord_t(P(R\cup\{q\})\!)=ord_t(P(R)\!)$ does not depend on $q$, which says that $R\in Sol(P)$, a contradiction. We get that  $q\geq k$ gets  determined by  
\begin{equation*}
ord_t(P(\varphi)\!)=a_{M_j}+\alpha(M_j,R)+d(M,R)q,
\end{equation*}
since $d(M,R)\in\mathbb{Z}_{>0}$ by \eqref{eq:def_M_R_infty}, which ends the proof.

Note that in general the value $q\geq k$ will be an element of $\mathbb{Q}$, unless $d(M,R)=1$ (e.g. the linear case), or if $d(M,R)\in\mathbb{Z}_{>1}$ it divides $ord_t(P(\varphi)\!)-a_{M_j}-\alpha(M_j,R)$. The case $q\in\mathbb{Q}\setminus\mathbb{Z}$ will make sense as a Puiseux (Hahn) tropical solution, which will be discussed in Prop. \ref{thm:H(P)_finite}. Again, in the linear case, there are no non-generic projections other than the faces touching the hyperplane at infinity.

In the case in which $S$ is not homogeneous, it is better to work with the condition $\#Sol_k(P)<\infty$, since this is better described by tropical algebraic conditions, and this implies also that $\#\mu(Sol(P)\!)<\infty$. We finish this section with a bound on the number of monomials of elementary $k$-solutions $\varphi\in Sol_k(P,[k-1])$, which may be helpful for computational aspects.

\begin{proposition}
    Any   solution   $\varphi\in Sol_k(P,[k-1])$ has at most $\min\{\max_{i\neq j}\{\#Supp(M_i)+\#Supp(M_j)\},ord(P)+1\}$ monomials.
\end{proposition}
\begin{proof}
We know from Prop. \ref{prop:ot_basis} that any $\varphi\in Sol_k(P,[k-1])$ has at most $k+1$ monomials.  Let $\varphi=R\cup B$ with $R=\{0\leq r_1<\cdots<r_s\leq k-1\}$ and $B\in\overline{\mathbb{Z}}_{\geq k}$ be a  solution, so there are $M_1,M_2\in S(\varphi)$ such that $ord_t(P(\varphi)\!)=ord_t(E_{M_i}(\varphi)\!)\neq\infty$ for $i=1,2$.
    
This means that for $i=1,2$, and $m_{i,j}\neq0$, there should exist a monomial $t^{r(j)}$ in $\varphi$ such that $r(j)\geq j$, so if we order $Supp(M_{1})=\{p_0<\cdots <p_b\}$ we get a sequence $r(p_0)\leq \cdots \leq r(p_b)$ with at  most  $\#Supp(M_{1})$ distinct elements. We do the same for $Supp(M_{2})=\{q_0<\cdots <q_c\}$ to get a sequence $r(q_0)\leq \cdots \leq r(q_c)$ with at  most  $\#Supp(M_{2})$ distinct elements. We have that $\varphi$ has at most $\#Supp(M_{1})+\#Supp(M_{2})$ elements (when all these monomials are different), and then we take the maximum over all the distinct pairs of monomials of $P$.
\end{proof}

 \begin{example}
    If $P$ is linear, this is, $S=\{e_0,\ldots,e_k\},$ then $\max_{i\neq j}\{\#Supp(M_i)+\#Supp(M_j)\}=2$, so we recover the original bound from \cite{GG26}.
\end{example}

 We have seen that there are two relevant sets of polynomial solutions for an \textsc{otde} $P$, there are two relevant sets of polynomial solutions, namely $\mu(Sol(P)\!)\subset Sol_k(P)$, together with decompositions $\mu(Sol(P)\!)=\mu(P,[k-1])\sqcup\mu(P,[k-1]^c)$ and $Sol_k(P)=Sol_k(P,[k-1])\sqcup Sol_k(P,[k-1]^c)$, and since $Sol_k(P,[k-1]^c)=\mu(P,[k-1]^c)$, we have 
\begin{equation}
\label{eq:difference_between_k_and_min}
    Sol_k(P)\setminus \mu(Sol(P)\!)=Sol_k(P,[k-1])\setminus\mu(P,[k-1]).
\end{equation}

This says that $\mu(Sol(P)\!)\subset Sol_k(P)$ only differ at the level of $Sol_k(P,[k-1])\setminus\mu(P,[k-1])$, which are finite sets if and only if $P\notin Y(S).$

\section{Tropical Laurent and Hahn solutions}
\label{Sec:App}
After the generalization of the theory of minimal solutions of \textsc{otde}s through the study of elementary $k$-solutions in  sections \ref{sec:Gen_prop_Min_Sol} and \ref{sect:F(P)}, in this one we  come to applications of this set of ideas to the theory of \textsc{oade}s.

We start with some motivation from the  theory of algebraic curve singularities. Given the germ $(C,0)$ of an algebraic curve singularity (defined over any field),  the semigroup of values $S(C,0)$   records the set of $t$-adic orders of the germs regular functions defined on $(C,0)$. See \cite{CamCas}.

With this concept in mind, for $\mathfrak{P}=\mathfrak{P}(y)$ an \textsc{oade}  we introduce   the set $
    S(\mathfrak{P},0):=ord_t(\{\Phi(t)\in\mathbb{C}[\![t^{\mathbb{R}}]\!]\::\:\mathfrak{P}(\Phi(t)\!)=0\})\subset \mathbb{R}$ of $t$-adic orders of (germs of)  formal Hahn power series solutions of $\mathfrak{P}$ at the point  $t_0=0\in\mathbb{C}$, as in \eqref{eq:OCDE_tadicord}. On the other hand, we  introduce the set $Sol_{\mathbb{B}[\![t^{\mathbb{R}}]\!],k}(P)$ of  tropical Hahn elementary $k$-solutions of a tropical polynomial $P$ of differential order $k$. The final goal of this Section is to show that if $P=trop(\mathfrak{P})$, then there is an inclusion 
\begin{equation}
    S(\mathfrak{P},0)\subset ord_t(Sol_{\mathbb{B}[\![t^{\mathbb{R}}]\!],k}(P)\!).
\end{equation}
Under certain conditions, the previous inclusion factors through the set of $t$-adic of the minimal solutions $S(\mathfrak{P},0)\subset ord_t(\mu(Sol_{\mathbb{B}[\![t^{\mathbb{R}}]\!]}(P)\!))\subset ord_t(Sol_{\mathbb{B}[\![t^{\mathbb{R}}]\!],k}(P)\!)$.

\subsection{From complex analysis to tropical differential algebra}
\label{Sec:fromCtoT}
We are looking for applications of these tropical methods to the theory of  ordinary algebraic differential equations (abbr. \textsc{oade}) $\mathfrak{P}$, so we will start this Section by discussing some aspects of the analytic theory, and we will slowly move to a more algebraic setting where both the set $S(\mathfrak{P},0)$ from \eqref{eq:OCDE_tadicord} can be defined, and the tropical methods can be applied. We will work over the Banach field $\mathbb{C}$, and we denote by $\mathcal{M}_{\mathbb{C}}$ the sheaf of meromorphic functions on  $\mathbb{C}$.

First, we consider  \textsc{oade}s $\mathfrak{P}$ which are algebraic, this is, they can be expressed as a differential polynomial  $\mathfrak{P}= \sum_{i=1}^s {\mathcal A}_{M_i} (z)E_{M_i}$  with   coefficients ${\mathcal A}_{M_i}$ which are meromorphic functions defined on some open region $U\subset\mathbb{C}$, this is, ${\mathcal A}_{M_i}\in \mathcal{M}_{\mathbb{C}}(U)$. Some points $t_0\in U$ may be singular in the sense that every coefficient function ${\mathcal A}_{M_i}(z)$ of $\mathfrak{P}$ has a pole at $t_0$.

In principle, the Painlevé property for the OCDE $\mathfrak{P}$ is concerned about the behavior  of the solutions  $\Phi(z)$ of $\mathfrak{P}$ with $\Phi\in \mathcal{M}_{\mathbb{C}}(U)$, with respect to their {\it singularities} (poles of $\Phi$).

If $\mathcal{M}_{\mathbb{C},0}$ denotes  the field of germs  of meromorphic functions at $t_0=0\in\mathbb{C}$, then the germ at $0$ map $\Phi\mapsto \Phi_0$ defines a map 
\begin{equation}
    \{\Phi\in \mathcal{M}_{\mathbb{C}}(U)\::\:\mathfrak{P}(\Phi)=0\}\xrightarrow[]{}\{\Psi\in \mathcal{M}_{\mathbb{C},0}\::\:\mathfrak{P}(\Psi)=0\},
\end{equation}

for every open  $U\subset\mathbb{C}$ satisfying $t_0=0\in U$. We denote  by $\mathcal{O}_{\mathbb{C},0}\subset \mathcal{M}_{\mathbb{C},0}$ the local ring  of  germs of holomorphic functions at $0\in\mathbb{C}$. We have inclusions $\mathcal{M}_{\mathbb{C},0}\subseteq \mathbb{C}(\!(t)\!)=\mathbb{C}[\![t^{\mathbb{Z}}]\!]$ (the field of formal Laurent power series) and $\mathcal{O}_{\mathbb{C},0}\subseteq\mathbb{C}[\![t]\!]= \mathbb{C}[\![t^{\mathbb{N}}]\!]$ (the ring of formal  power series). This in turn induces a diagram of inclusions 
\begin{equation}
    \xymatrix{\{\Phi\in \mathcal{M}_{\mathbb{C},0}\::\:\mathfrak{P}(\Phi)=0\}\ar[r]&Sol_{\mathbb{C}[\![t^{\mathbb{Z}}]\!]}(\mathfrak{P}):=\{\Phi\in \mathbb{C}[\![t^{\mathbb{Z}}]\!]\::\:\mathfrak{P}(\Phi)=0\}\\
    \{\Phi\in \mathcal{O}_{\mathbb{C},0}\::\:\mathfrak{P}(\Phi)=0\}\ar[r]\ar[u]&Sol(\mathfrak{P}):=\{\Phi\in \mathbb{C}[\![t^{\mathbb{N}}]\!]\::\:\mathfrak{P}(\Phi)=0\}\ar[u]\\}
\end{equation}

So the set $Sol_{\mathbb{C}[\![t^{\mathbb{Z}}]\!]}(\mathfrak{P})=\{\Phi\in \mathbb{C}[\![t^{\mathbb{Z}}]\!]\::\:\mathfrak{P}(\Phi)=0\}$ is an algebraic generalization of the set of germs of meromorphic functions at $0\in\mathbb{C}$ which are also solutions of $\mathfrak{P}$. In order to make the notation less convoluted, sometimes we will denote $LSol(\mathfrak{P})=Sol_{\mathbb{C}[\![t^{\mathbb{Z}}]\!]}(\mathfrak{P})$.

From the same  type of generalizations, we can consider instead algebraic equations $\mathfrak{P}= \sum_{i=1}^s {\mathcal A}_{M_i} E_{M_i}$ as in \eqref{eq:ODE} with Laurent series coefficients ${\mathcal A}_i=t^{a_i} \sum_{j\geq 0} {\mathcal A}_{ij} t^j\in\mathbb{C}[\![t^{\mathbb{Z}}]\!]$, $a_i\in\mathbb{Z}$ and $\ {\mathcal A}_{i0}\in \mathbb{C}^*$, this is, we can work with elements of the differential ring $\mathfrak{P}\in\mathbb{C}[\![t^{\mathbb{Z}}]\!]\{y\}$. 

Let $ord_t:\mathbb{C}[\![t^{\mathbb{Z}}]\!]\xrightarrow[]{}\overline{\mathbb{Z}}$ be the $t$-adic valuation (see diagram \ref{diag:relevant}). Note that this interpretation of our algebraic generalization of germs of solutions is valid only at $t_0=0\in\mathbb{C}$, thus  this is the only {\it point} that may appear as a singularity of $\mathfrak{P}$, and this happens when $ord_t({\mathcal A}_i)=a_i\in\mathbb{Z}_{<0}$ for all $i=1,\ldots,s$.

If we extend the standard terminology to formal power series solutions $\Phi\in LSol(\mathfrak{P})$, these  can be either
\begin{enumerate}
    \item holomorphic at zero, in which case they are either regular ($ord_t(\Phi)=0$) or they have a zero at $0\in\mathbb{C}$ of some order $ord_t(\Phi)>0$, or
    \item have a pole at $0\in\mathbb{C}$ of some order $-ord_t(\Phi)>0$.
\end{enumerate}

While going from analysis to algebra make us lose the notion of local convergence (around the origin) of the germs of meromorphic (holomorphic) solutions,  it makes possible the application of algebraic and tropical methods. 

The theory of tropical differential algebraic geometry studies from a combinatorial and algebraic perspectives, the support sets these algebraic generalizations of germs of  solutions of equations $\mathfrak{P}\in\mathbb{C}[\![t^{\mathbb{Z}}]\!]\{y\}$ at the point $t_0=0\in\mathbb{C}$, disregarding whether or not these formal solutions represent germs of honest functions in an open neighborhood of $t_0=0$.  Concretely, tropical methods appear when we consider the interaction between, on one hand, the image of the support map $Supp:LSol(\mathfrak{P})\xrightarrow[]{}\mathcal{P}(\mathbb{Z})$, and on the other hand, the set of Laurent tropical solutions of the tropicalization $P=trop(\mathfrak{P})$. 

From now on we will restrict our attention to different algebraic generalizations of germs  $\Phi$ of meromorphic functions at $0\in\mathbb{C}$ which are at the same time solutions of some $\mathfrak{P}\in\mathbb{C}[\![t^{\mathbb{Z}}]\!]\{y\}$.  In Section \ref{sec:trop_LS} we will start with tropical Laurent solutions, and later in Sect. \ref{sec:ths}, we will introduce  more general germs of functions at $0\in\mathbb{C}$ which are also solutions of $\mathfrak{P}$.

We summarize the relevant objects objects that we will consider in this Section and their interactions into a the following commutative diagram (see also diagram \eqref{eq:ord_comm_with_supp})

\begin{equation}
\label{diag:relevant}
\xymatrix{    
 && \mathcal{H}_{\mathbb{C},0}\ar[r]&\mathbb{C}[\![t^{\mathbb{R}}]\!]\ar[d]_{Supp}\ar[ddr]^{ord_t}&\\
   & \mathcal{M}_{\mathbb{C},0}\ar[r]\ar[ur]&\mathbb{C}[\![t^{\mathbb{Z}}]\!]\ar[d]_{Supp}\ar[ddr]^{ord_t}\ar[ur]&\mathbb{B}[\![t^{\mathbb{R}}]\!]\ar[dr]_{ord_t}&\\
  \mathcal{O}_{\mathbb{C},0}\ar[r]\ar[ur] & \mathbb{C}[\![t^{\mathbb{N}}]\!]\ar[d]_{Supp}\ar[ddr]^{ord_t}\ar[ur]&\mathbb{B}[\![t^{\mathbb{Z}}]\!]\ar[dr]_{ord_t}\ar[ur]&&\mathbb{T}\\
    &\mathbb{B}[\![t^{\mathbb{N}}]\!]\ar[ur]\ar[dr]_{ord_t}&&\overline{\mathbb{Z}}\ar[ur]\\
   & &\overline{\mathbb{N}}\ar[ur]&&\\}
\end{equation}

\subsection{Tropical Laurent   solutions}
\label{sec:trop_LS}
In this part we generalize Section \ref{sect:TDA} from $\Gamma=\mathbb{N}$ to $\Gamma=\mathbb{Z}$ to  give the definition of (tropical) Laurent  series solution of \eqref{eq:otde}, which is a special case of (tropical) Hahn series solution, to be introduced in Section \ref{sec:ths}, and which corresponds to the choice $\Gamma=\mathbb{R}$.  See \eqref{diag:relevant} for a  diagram depicting all the relevant interactions between these objects. 

Recall that $\mathbb{B}[\![t^{\mathbb{Z}}]\!]$ denotes the semiring consisting of  objects of the form $\varphi=\sum_{i\in A}1_{\mathbb{B}}t^i$ where $A\subset \mathbb{Z}$ is well-ordered. The formal derivation of series  $\tfrac{d}{dt}$ extends to an operator $\mathbb{B}[\![t^{\mathbb{Z}}]\!]\xrightarrow[]{}\mathbb{B}[\![t^{\mathbb{Z}}]\!]$, as well as the commutative diagram: 
\begin{equation*}
    \xymatrix{\mathbb{B}[\![t^{\mathbb{Z}}]\!]\ar[r]^{dev_P}\ar[dr]_{trop_P}&\mathbb{B}[\![t^{\mathbb{Z}}]\!]\ar[d]^{ord_t}\\
    &\overline{\mathbb{Z}}.}
\end{equation*}  

Once again, the isomorphism $Supp:\mathbb{B}[\![t^{\mathbb{Z}}]\!]\xrightarrow[]{}\mathcal{P}_{wo}(\mathbb{Z})$ from Example \ref{ex:BPS} can be promoted to an isomorphism of semirings equipped with a derivation $\tfrac{d}{dt}$.  If we continue with our  Convention \ref{conv:shape_of_P}, then Prop. \ref{prop:equaltozerosols} remains valid, and we have 
\begin{definition}
\label{def:Laurent_tropsols}
    For $P=P(y)=\sum_{M_i\in S}t^{a_{M_i}}E_{M_i}$ as in \eqref{eq:otde}, we define 
     \begin{equation}
        Sol_{\mathbb{B}[\![t^{\mathbb{Z}}]\!]}(P)=\{\varphi\in\mathbb{B}[\![t^{\mathbb{Z}}]\!]\::\:ord_t(P(\varphi)\!)\text{``} = \text{'' } 0_{\mathbb{T}}\}
     \end{equation}
     the set of tropical Laurent solutions. We will write $LSol(P)=Sol_{\mathbb{B}[\![t^{\mathbb{Z}}]\!]}(P)$ in order to make the notation less heavy.
\end{definition}

Now we repeat the procedure from Section \ref{sec:ekseries_and_sols} with $\Gamma=\mathbb{Z}$ in place of $\Gamma=\mathbb{N}$. If we write $\mathbb{Z}=[k-1]\sqcup[k-1]^c$, and if $\varphi\in \mathbb{B}[\![t^{\mathbb{Z}}]\!]$ satisfies $ord_t(\varphi)\in[k-1]$, then $\varphi_{[k-1]^c}=\varphi_{\mathbb{Z}_{\geq k}}$. 

If we convene that now we take our complements with respect to $\mathbb{Z}$ instead of with respect to $\mathbb{N}$, then the Definition \ref{def:eksN} of $\pi_k$ can be extended naturally to $\pi_k:\mathbb{B}[\![t^{\mathbb{Z}}]\!]\xrightarrow[]{}\mathbb{B}[\![t^{\mathbb{Z}}]\!]$, afterwards we extend the definition of elementary  $k$-series to tropical Laurent series as the image $Im(\pi_k)\subset \mathbb{B}[\![t^{\mathbb{Z}}]\!]$; now these are of the form $\varphi=R\cup B$ with  

\begin{enumerate}
         \item $ord_t(\varphi)\notin[k-1]$, then $R=\emptyset$ and  $B=\{q\}$ with $q\in\mathbb{Z}\setminus\{0,\ldots,k-1\}$,
         \item $ord_t(\varphi)\in[k-1]$, then $\emptyset\neq R\subset [k-1]$ and the sub-cases 
         \begin{enumerate}
             \item $B=\emptyset$,
             \item $B=\{q\}$ with $k\leq q<\infty$.
         \end{enumerate}
    \end{enumerate}

Endow $LSol(P)$  with the order induced by inclusion of subsets of $\mathcal{P}(\mathbb{Z})$, and denote by $\mu(LSol(P)\!)$ the minimal elements of the poset $LSol(P)$. If we consider the set $LSol_k(P)$ consisting of elementary Laurent $k$-solutions of $P$, since Lemma \ref{lemma_computations} remains valid,  then we have a map $\pi_k:LSol(P)\xrightarrow[]{}LSol_k(P)$ and $\mu(LSol(P)\!)\subset LSol_k(P)$.

Note that the sets $LSol_k(P,[k-1])=\{\varphi\in LSol_k(P)\!)\::\:ord_t(\varphi)\in[k-1]\}$ and $LSol_k(P,[k-1]^c)=\{\varphi\in LSol_k(P)\::\:ord_t(\varphi)\notin[k-1]\}$ make sense in this new context, and yield the decomposition
 
\begin{equation*}
        LSol_k(P)=LSol_k(P,[k-1])\sqcup LSol_k(P,[k-1]^c),
    \end{equation*}
    where $LSol_k(P,[k-1])\subset LSol_k(P)$ is exactly the same as it was for the case $Sol_k(P,[k-1])\subset Sol_k(P)$, so if $Y(S)\subset(\mathbb{T}^*)^s$ is defined as in \eqref{eq:deg_sols}, then $ LSol_k(P,[k-1])$ is finite if and only if $P\in\mathbb{Z}_S\setminus Y(S)$ by Thm. \ref{thm:crit_fin}.

    So there is only left to check the cardinality of the set $LSol_k(P,[k-1]^c)\subset LSol_k(P)$. For this, we consider a support set $S=\{M_1,\ldots,M_s\}$ of order $k$ as usual, and recall that $\Delta_{S,\emptyset}|_P(w)=min_j\{|M_j|_k+a_{M_j}+deg(M_j)w\}$ for $P=(a_M)\in\mathbb{Z}_S$. We have the following generalization of Prop. \ref{prop:poly}.

\begin{proposition}
\label{prop_neg_ord_sols}
    If $\varphi\in  LSol_k(P)\setminus Sol_k(P)\!)$, then $\varphi=t^q$ for some $q\in \mathbb{Z}_{<0}$. Moreover, we have  a bijection of sets 
    \begin{equation}
        \begin{aligned}
            LSol_k(P)\setminus Sol_k(P)&\xrightarrow[]{\sim} V(\Delta_S|_P(w)\!)\cap\mathbb{Z}_{<-k}\\
            q&\mapsto q-k
        \end{aligned}
    \end{equation}
\end{proposition}
\begin{proof}
If $\varphi\in  LSol(P)\setminus Sol(P)$, then $q=ord_t(\varphi)\in\mathbb{Z}_{<0}$, and since $ord_t(\varphi)\notin[k-1]$, we have that $\varphi\in  LSol_k(P)\setminus Sol_k(P)$ if and only if $\varphi=t^q$ with $q<0$.

Now, the computation of point (1) from Lemma \ref{lemma_computations} remains valid, so we can apply the Prop. \ref{prop:poly}, with the difference that $t^q$ is a minimal solution  according to the new Definition \ref{def:Laurent_tropsols}, if and only if $q-k\in V(\Delta_{S,\emptyset}|_P(w)\!)$, which finishes the proof. 
\end{proof}

\begin{remark}
    In \cite{DG}, D. Grigoriev introduced Laurent solutions for \textsc{otde}s, which for our case ($n=1$ differential variables) are solutions either of the form $\varphi\in\mathbb{B}[\![t^{\mathbb{N}}]\!]$ or $\varphi=t^q$ with $q\in\mathbb{Z}_{<0}$, which as we saw are just the elements of the set $LSol_k(P)\setminus Sol_k(P)\!)$.
\end{remark}

We denote by $S_{max}\subset S$ the subset consisting of monomials of maximal degree of $S$. We have the following analogue of Prop. \ref{prop:nothcase}.

\begin{proposition}
\label{prop:nothcase_laurent}
    If $P\in\mathbb{Z}_S$, we write $P_{max}=\sum_{M\in S_{max}}t^{a_M}E_M$. Then     $\#(LSol_k(P)\setminus Sol_k(P)\!)=\infty$ if and only if 
    \begin{equation}
        \label{eq:condition_poles}
        P_{max}\in V(\Delta_{S_{max}}(x_M\::\:M\in S_{max},w)|_{w=0}).
    \end{equation}
\end{proposition}

\begin{proof}
By Proposition \ref{prop_neg_ord_sols}, it is enough to study the behavior of  the monomials  from $\Delta_S|_P(w)$, which are  Euclidean rays $w\mapsto C+dw$ for $C$ a constant and $d\in deg(S)=\{d_1<\cdots<d_m\}$ a positive integer.

For $w<\!\!<0$, the value $\Delta_S|_P(w)$ will be taken at a monomial $|M_j|_k+a_{M_j}+deg(M_j)w$ with $deg(M_j)\in S_{max}$, and the only way in which $\Delta_S|_P(w)$ could vanish tropically in a ray $(-\infty,q]$ is if there exists another monomial $|M_b|_k+a_{M_b}+deg(M_b)w$ defining the same Euclidean ray, in particular $deg(M_b)\in S_{max}$. This happens if and only if condition \eqref{eq:condition_poles} is satisfied.
\end{proof}

We now present the main result of this Section, which is  a description of the locus $LSing(S)=\{P\in\mathbb{Z}_S\::\:\#LSol_k(P)=\infty\}$ from \eqref{eq:sing_laurent}, which can be seen as an extension of Theorem \ref{thm:crit_fin}. 

\begin{theorem}
\label{thm:crit_fin_mero}
Let $S=\{M_1,\ldots,M_s\}$ with $Y(S)\subset(\mathbb{T}^*)^s$  as in \eqref{eq:deg_sols}, and $S_{min}$, $S_{max}$ the set of minimal and maximal degree monomials, respectively. Then $\#LSol_k(P)<\infty$ if and only if 
    \begin{equation*}
        P\notin V(\Delta_{S_{min},\emptyset}|_{w=0},\Delta_{S_{max},\emptyset}|_{w=0})\cap Y(S).
    \end{equation*}
\end{theorem}

\begin{proof}
    We have that $LSol_k(P)$ for $P\in\mathbb{Z}_S$ is finite if and only if both $LSol_k(P,[k-1]^c)$ and $LSol_k(P,[k-1])$ are finite, since $LSol_k(P)=LSol_k(P,[k-1])\sqcup LSol_k(P,[k-1]^c)$.  Now, $LSol_k(P,[k-1])$ is finite if and only if $P\notin Y(S)$, and $LSol_k(P,[k-1]^c)$ is  the disjoint union of $\{ord_t(\varphi)\geq k\}$ and $\{ord_t(\varphi)<0\}$, and the result follows now from Theorem \ref{thm:crit_fin} and Prop. \ref{prop:nothcase_laurent}.
\end{proof}

Together, Props. \ref{prop:poly} and \ref{prop_neg_ord_sols} give us a new bijection
\begin{equation}
        \begin{aligned}
            LSol_k(P,[k-1]^c)&\xrightarrow[]{\sim} V(\Delta_S|_P(w)\!)\cap(\mathbb{Z}\setminus\{-1,-2,\ldots,-k\})\\
            q&\mapsto q-k
        \end{aligned}
    \end{equation}

\begin{remark}
\label{rem:coro_2}
    Note that if $S$ is homogeneous of degree $d$, then $LSing(S)=Sing(S)$, since $S_{min }=S_{max} =\{d\}$. In particular, Cor. \ref{prop:hcase} generalizes to 

\begin{equation*}
    \mu(P,[k-1]^c)=\begin{cases}
        \emptyset, &\text{ if }P\notin V(\Delta_S(x_{M_1},\ldots,x_{M_s},z)|_{z=0}),\\
        \mathbb{Z}\setminus\{0,\ldots,k-1\},&\text{ otherwise},\\
    \end{cases}
\end{equation*}
and since $LSol_k(P,[k-1])$ remains the same, no new tropical Laurent phenomena appear for the homogeneous case at the level of elementary $k$-solutions.
\end{remark}

If $S$ is not homogeneous, since $Supp(\Delta_{S_{min},\emptyset})=S_{min}$ and $Supp(\Delta_{S_{max},\emptyset})=S_{max}$,  we can compute $V(\Delta_{S_{min},\emptyset})\subset (\mathbb{T}^*)^{\#S_{min}}$ (This only works for $\#S_{0}>1$, otherwise $V(\Delta_{S_{min},\emptyset})=\emptyset$) and similarly $V(\Delta_{S_{max},\emptyset})\subset (\mathbb{T}^*)^{\#S_{max}}$ and consider $V(\Delta_{S_{min},\emptyset})\times \mathbb{T}^{\#(S\setminus S_{min})}\subset \mathbb{T}^s$ and $V(\Delta_{S_{max},\emptyset})\times \mathbb{T}^{\#(S\setminus S_{max})}\subset \mathbb{T}^s$. Then 
    \begin{equation*}
        LSing(S)=\begin{cases}
          \emptyset,&\text{ if }\#S_{min}=\#S_{max}=1,\\
        V(\Delta_{S_{min},\emptyset})\times \mathbb{T}^{\#(S\setminus S_{min})},&\text{ if }\#S_{min}>1,\#S_{max}=1\\
        V(\Delta_{S_{max},\emptyset})\times \mathbb{T}^{\#(S\setminus S_{max})},&\text{ if }\#S_{min}=1,\#S_{max}>1\\
        V(\Delta_{S_{min},\emptyset})\times \mathbb{T}^{\#(S\setminus S_{min})}\cap V(\Delta_{S_{max},\emptyset})\times \mathbb{T}^{\#(S\setminus S_{max})},&\text{ if }\#S_{min},\#S_{max}>1\\
        \end{cases}
    \end{equation*}

Recall that  we used the tropical prevariety $Y(S)\cap V(\Delta_{S_{min}}(x_M\::\:M\in S_{min})\!)=X(S)\subset\mathbb{T}^s$ to extend the concept of regularity from Def. \ref{def:ext_regular}, which itself extends \cite[Definition 8]{GG26}. Now let $X_L(S)\subset\mathbb{T}^s$ be the tropical prevariety $X_L(S)=X(S)\cap V(\Delta_{S_{max},\emptyset}|_{w=0})$.

\begin{definition}
    We say that $P=(a_M)\in \mathbb{Z}_S$ is Laurent regular if $P\notin X_L(S)$.
\end{definition}

We finish this part with the following extension of Corollary \ref{coro:coro1}, which considers also minimal tropical Laurent solutions, and extends the concept of regularity to Laurent regularity.  

\begin{corollary}
\label{cor:2}
We have that $P$ is Laurent regular if and only if $LSol_k(P)=LSol_k(P,[k-1])\sqcup LSol_k(P,[k-1]^c)$ satisfies: 
\begin{enumerate}
        \item any $\varphi\in LSol_k(P,[k-1])$ is of the form $R\cup\{q\}$, with $\emptyset\neq R\subset [k-1]$ and $q=q(R)\geq k$ varying in a finite set,
        \item $LSol_k(P,[k-1]^c)\cong V(\Delta_S|_P(w)\!)\cap[\mathbb{Z}\setminus\{-1,-2,\ldots,-k\}]$ is a finite set, which is empty if $S$ is homogeneous. Here $\Delta_{S,\emptyset}|_P(w)=min_j\{|M_j|_S+a_{M_j}+deg(M_j)w\}$. 
    \end{enumerate}
\end{corollary}
\begin{proof}
    This follows from Corollary \ref{coro:coro1} and Theorem \ref{thm:crit_fin_mero}. 
\end{proof}

So, in order to start this study, we need to check $Sol_k(P,[k-1])$ (which is independent from the notion  of Laurent solution) and it is done in Section \ref{sec.generic_solutions}, and $LSol_k(P,[k-1]^c)$ only when $S$ is not homogeneous, which is done in  Section   \ref{sec:ordatleastk}.

\subsection{Tropical Hahn solutions}
\label{sec:ths}
In the previous section, motivated by the application of tropical methods towards the study of the Painlevé property an \textsc{oade} $\mathfrak{P}$ at the point $t_0=0\in\mathbb{C}$, and the set $S(\mathfrak{P},0)$ from \eqref{eq:OCDE_tadicord}, we extended the set of tropical solutions of an \textsc{otde} $P$ from $\mathbb{B}[\![t^{\mathbb{N}}]\!]$ to $\mathbb{B}[\![t^{\mathbb{Z}}]\!]$. However, the 
Painlevé property   not only considers the $t$-adic order of the (germs at $t_0=0$ of) meromorphic  solutions $\Phi$ for equations $\mathfrak{P}$, but it also considers the possibility that the $t$-adic order of these germs of solutions may be a rational number. 

The theory of tropical solutions of an \textsc{otde} $P$  in $\mathbb{B}[\![t^{\Gamma}]\!]$ that arises from the cases $\Gamma=\mathbb{N}$ (formal power series) and $\Gamma=\mathbb{Z}$ (formal Laurent series)  differs little from the theory that arises from the cases
$\Gamma=\mathbb{Q}$ (formal Puiseux series) and $\Gamma=\mathbb{R}$ (formal Hahn series), due to the discrete-non discrete dichotomy. Thus we shall skip the rational case and go directly to the case $\Gamma=\mathbb{R}$.

Let us denote by  $\mathcal{H}_{\mathbb{C},0}\subset \mathbb{C}[\![t^{\mathbb{R}}]\!]$  the set consisting of germs of {\it Hahn functions} at $t_0=0\in\mathbb{C}$ (elements  $\Phi\in \mathbb{C}[\![t^{\mathbb{R}}]\!]$ that converge in an open neighborhood around $0$).  As before, the set $$HSol(\mathfrak{P})=Sol_{\mathbb{C}[\![t^{\mathbb{R}}]\!]}(\mathfrak{P})=\{\Phi\in \mathbb{C}[\![t^{\mathbb{R}}]\!]\::\:\mathfrak{P}(\Phi)=0\}$$ is the algebraic generalization of the set of germs of Hahn functions  which are also solutions of $\mathfrak{P}$. Since $\Gamma\subset\mathbb{R}$ for $\Gamma=\mathbb{N}$ and $\Gamma=\mathbb{Z}$, this set contains as subsets our previous sets of formal germs of solutions at $0\in\mathbb{C}$. See diagram \eqref{diag:relevant}.

 It was shown in \cite{GS} that if $\Phi\in \mathbb{C}[\![t^{\mathbb{R}}]\!]$ is a solution to some $\mathfrak{P}\in\mathbb{C}[\![t]\!]\{y\}$, then $Supp(\Phi)$ is countable. Thus, we will restrict ourselves to the subfield of real Hahn series with countable support, which we also denote by $\mathbb{C}[\![t^{\mathbb{R}}]\!]$. 

We now introduce tropical analogues of these concepts: recall that $\mathbb{B}[\![t^{\mathbb{R}}]\!]$ denotes the semiring consisting of  objects of the form $\varphi=\sum_{i\in A}1_{\mathbb{B}}t^i$ where $A\subset \mathbb{R}$ is well-ordered (and countable).

\begin{definition}
\label{def:Hahn_tropsols}
    For $P=P(y)=\sum_{M_i\in S}t^{a_{M_i}}E_{M_i}$ as in \eqref{eq:otde}, we denote by 
     \begin{equation}
         HSol(P)=Sol_{\mathbb{B}[\![t^{\mathbb{R}}]\!]}(P)=\{\varphi\in\mathbb{B}[\![t^{\mathbb{R}}]\!]\::\:ord_t(P(\varphi)\!)\text{``} = \text{'' } 0_{\mathbb{T}}\}
     \end{equation}
     the set of tropical Hahn solutions. 
\end{definition}

We repeat the procedure from Section \ref{sec:trop_LS}, with $\Gamma=\mathbb{R}$ in place of $\Gamma=\mathbb{Z}$.  If we write $\mathbb{R}=[k-1]\sqcup[k-1]^c$, and if $\varphi\in \mathbb{B}[\![t^{\mathbb{R}}]\!]$ satisfies $ord_t(\varphi)\in[k-1]$, then it can be expressed uniquely as 
\begin{equation}
\varphi=\varphi_{\mathbb{N}}+\varphi_{\mathbb{R}\setminus\mathbb{N}},\quad\text{with }ord_t(\varphi)=ord_t(\varphi_{\mathbb{N}})<ord_t(\varphi_{\mathbb{R}\setminus\mathbb{N}}),
\end{equation}
where $0\neq\varphi_{\mathbb{N}}\in\mathcal{P}(\mathbb{N})$ and $\varphi_{\mathbb{R}\setminus\mathbb{N}}\in \mathcal{P}(\mathbb{R})\setminus\mathcal{P}(\mathbb{N})$. If additionally $\varphi\notin \mathcal{P}([k-1])$, then $\varphi_{\mathbb{R}\setminus\mathbb{N}}\neq0.$

If we convene that now we take our complements with respect to $\mathbb{R}$ instead of  $\mathbb{Z}$, we extend Definition \ref{def:eksN} of $\pi_k$ to $\pi_k:\mathbb{B}[\![t^{\mathbb{R}}]\!]\xrightarrow[]{}\mathbb{B}[\![t^{\mathbb{R}}]\!]$ as follows:

\begin{equation}
        \pi_k(\varphi)=\begin{cases}
       t^{ord_t\varphi} &\text{if }ord_t(\varphi)\notin[k-1],\\
            \varphi,&\text{if }\varphi\in \mathcal{P}([k-1]),\\
            \varphi_{\mathbb{N}}\cap[0,min\{k-1,\lfloor ord_t(\varphi_{\mathbb{R}\setminus\mathbb{N}})\rfloor\}]+t^{ord_t\varphi_{\mathbb{R}\setminus\mathbb{N}}},&\text{otherwise.}
        \end{cases}
    \end{equation}

We extend the definition of elementary Laurent $k$-series to tropical Hahn series as the image $Im(\pi_k)\subset \mathbb{B}[\![t^{\mathbb{R}}]\!]$; these are of the form $\varphi=R\cup B$ with  

\begin{enumerate}
    \item $ord_t(\varphi)\notin[k-1]$, then $R=\emptyset$ and  $B=\{\lambda\}$ with $\lambda\in\mathbb{R}\setminus\{0,\ldots,k-1\}$,
    \item $ord_t(\varphi)\in[k-1]$, then we have the sub-cases 
         \begin{enumerate}
    \item 
    $\varphi\in \mathcal{P}([k-1])$, then $R=\varphi$ and $B=0$,
    \item  $\varphi=R\cup\{\lambda\}$, with $\lambda\in\mathbb{R}_{\geq0}\setminus[k-1]$ and $\emptyset\neq R\subset[0,min\{k-1,\lfloor \lambda\rfloor\}]$.
\end{enumerate}
\end{enumerate}

Similarly, we have an identification 
\begin{equation}
    \begin{aligned}
        \pi_k(\mathbb{B}[\![t^{\mathbb{R}}]\!])&\xrightarrow[]{}\{(\emptyset,\infty)\neq(R,\lambda)\in\mathcal{P}([k-1])\times \mathbb{T}\setminus[k-1]\::\:R\subset\mathcal{P}(min\{k-1,\lfloor \lambda\rfloor\})\}\\
        R\cup B&\mapsto (R,ord_t(B)\!)
    \end{aligned}
\end{equation}
thus, sometimes we will use the notation $\varphi=(R,\lambda)$ for $\varphi\in \pi_k(\mathbb{B}[\![t^{\mathbb{R}}]\!])$, or even  $\varphi=R\cup \lambda$.

In order to prove the generalization of Lemma \ref{lemma_computations} for Hahn series, we need   a preliminary  result.
\begin{lemma}
\label{lemma:new_Hahn}
Let $\varphi_1+\varphi_2=\varphi\in \mathbb{B}[\![t^{\mathbb{R}}]\!]$, where $0\neq\varphi_1\in\mathcal{P}(\mathbb{N})$ and $0\neq\varphi_2\in \mathcal{P}(\mathbb{R})\setminus\mathcal{P}(\mathbb{N})$ satisfying $ord_t(\varphi_{2})>ord_t(\varphi)=ord_t(\varphi_{1})\in[k-1]$.

Set $R:=\varphi_1\cap[0,min\{k-1,\lfloor ord_t(\varphi_2)\rfloor\}]$, then for every $M\in\mathbb{N}^{[k]}$ we have
\begin{equation*}
        ord_t(E_M(\varphi)\!)=ord_t(E_M(R\cup\{ord_t(\varphi_2)\})\!)=
        |M|_{R}(ord_t(\varphi_2)\!),
    \end{equation*}
where $|M|_{R}(z)$ (respectively $|M|_{R}(q)$) is given in  Definition \ref{def:values_of_evaluations} (respectively \eqref{eq:def_M_R_infty}).
    \end{lemma}
\begin{proof}
We need to compute $ord_t(\frac{d^j\varphi}{dt^j})$ for all $j\in\{0,\ldots,k\}$. Let $\emptyset\neq \varphi_1\cap\{0,\ldots,k-1\}=\{0\leq r_1<\cdots<r_s\leq k-1\}$, and let $\lambda=ord_t(\varphi_2)$, so $ord_t(\varphi_1)=r_1< \lambda.$

 If $\lambda\geq k$, then $min\{k-1,\lfloor ord_t(\varphi_2)\rfloor\}=k-1$, and $\varphi_1\cap\{0,\ldots,k-1\}=R$ and $B=\lambda$, so  $ord_t(E_M(\varphi)\!)=ord_t(E_M(R\cup B)\!)$ by Lemma \ref{lemma:explicit_computations}.

 If $\lambda< k$, then $r_1\leq \lfloor\lambda\rfloor\leq k-1$ and $min\{k-1,\lfloor ord_t(\varphi_2)\rfloor\}=\lfloor\lambda\rfloor$.  For $j\in\{0,\ldots,k\}$ and $\emptyset\neq \varphi_1\cap\{0,\ldots,k-1\}$ as above, we have  
    \begin{equation}
    \label{eq:gen_comp}
        ord_t(\frac{d^j\varphi}{dt^j})=\begin{cases}
            r_i-j,&\text{ if }r_{i-1}<j\leq r_i\leq \lfloor\lambda\rfloor,\\
            \lambda-j,&\text{ if }\lfloor\lambda\rfloor+1\leq j\\
        \end{cases}
    \end{equation}
    so the evaluation $ord_t(E_M(\varphi)\!)$ for  $M=\sum_j^km_je_j$ only depends on $R=\varphi_1\cap\{0,\ldots,min\{k-1,\lfloor\lambda\rfloor\}\}=\{0\leq r_1<\cdots<r_a\}$, and this is just \begin{equation*}
        ord_t(E_M(R\cup\lambda)\!)=\sum_{j=0}^{r_1}m_{j}[r_1-j]+\cdots+\sum_{j=r_{a-1}+1}^{r_a}m_{j}[r_a-j]+\sum_{j=r_a+1}^{k}m_{j}[\lambda-j]=
        |M|_R(\lambda),
    \end{equation*}
which ends the proof.
\end{proof}

Endow $HSol(P)$  with the order induced by inclusion of subsets of $\mathcal{P}(\mathbb{R})$, and denote by $\mu(HSol(P)\!)$ the minimal elements of the poset $HSol(P)$.  We have  the following extension of 
Prop. \ref {prop:ot_basis} (the proof is the same).

\begin{theorem}
\label{prop:ot_basis_Hahn}
Let $P$ be an \textsc{otde} as in \eqref{eq:otde}.
    Any $\varphi\in HSol(P)$ contains a minimal solution.
\end{theorem}
\begin{proof}
    We have that Lemma \ref{lemma_computations} is now valid in this context thanks to Lemma \ref{lemma:new_Hahn}, if we consider the set $HSol_k(P)$ consisting of elementary Hahn $k$-solutions of $P$, then $\mu(HSol(P)\!)\subset HSol_k(P)$, and we have a map $\pi_k:HSol(P)\xrightarrow[]{}HSol_k(P)$.
\end{proof}

Again, the sets $HSol_k(P,[k-1])=\{\varphi\in HSol_k(P)\::\:ord_t(\varphi)\in[k-1]\}$ and $HSol_k(P,[k-1]^c)=\{\varphi\in HSol_k(P)\::\:ord_t(\varphi)\notin[k-1]\}$ make sense in this new context, and yield the decomposition $HSol(P)=HSol_k(P,[k-1])\sqcup HSol_k(P,[k-1]^c).$ 
  
  For $HSol_k(P,[k-1])$, what changes from the Laurent case $\Gamma=\mathbb{Z}$ to the Hahn case $\Gamma=\mathbb{R}$ is that there may be solutions $R\cup\lambda$ with $\lambda\in\bigcup_{a=0}^{k-1}(a,a+1)$ and $\emptyset\neq R\subset[0,\lfloor\lambda\rfloor]$. 
  
  What we shall do now is, for $a=0,\ldots,k-1$  and $\emptyset\neq R\subset[0,a]$ fixed, we will define monomials $|M|_{R,a}(z)$   using \eqref{eq:gen_comp}: given $M=\sum_{j=0}^km_je_j\in \mathbb{N}^{[k]}$, we denote by $M_{\leq a}=\sum_{j=0}^{a}m_je_j$ and $M_{\geq a+1}=\sum_{j=a+1}^{k}m_je_j$.  Now we define $|M|_{R,a}(z)=|M_{\leq a}|_R(\infty)+|M_{\geq a+1}|_{\emptyset}(z)$ and \begin{equation}
  \label{eq:Hahn_new_disc}
        \Delta_{S,R,a}:=\bigoplus_{M\in S}|M|_{R,a}(z)\odot x_M\in\mathbb{T}[x_{M_1},\ldots,x_{M_s},z].
    \end{equation}

  If we define for $a=0,\ldots,k-2$ the interval $I_a=(a,a+1)\subset\mathbb{R}$, and $I_{k-1}=(k-1,\infty)\subset\mathbb{R}$, then $\mathbb{R}_{\geq0}\setminus[k-1]=\bigsqcup_{a=0}^{k-1}I_a$, and  we  have an analogue of Proposition \ref{prop:discs}. 

\begin{proposition}
\label{prop:discs_Hahn}
    Let $S$ be a support set,   $a\in[k-1]$, and  $\emptyset\neq R\in \mathcal{P}([a])$. Let $R\cup \lambda=\varphi\in Im(\pi_k)$ with $\lambda\in I_a$. Then   $\varphi\in HSol_k(P)$ for $P\in\mathbb{Z}_S$ if and only if $(P,\lambda)\in V(\Delta_{S,R,a})$.
\end{proposition}
    \begin{proof}
        By Lemma \ref{lemma:new_Hahn}, for $M\in \mathbb{N}^{[k]}$, $a=0,\ldots,k-2$ and  $\lambda\in(a,a+1)$, we have $\lfloor\lambda\rfloor=a$, so: 

    \begin{equation*}
ord_t(E_M(R\cup\lambda)\!)=|M|_R(\lambda)=|M_{\leq a}|_R(\infty)+|M_{\geq a+1}|_{\emptyset}(\lambda),
    \end{equation*}
    where $|M|_R(q)$ comes from \eqref{eq:def_M_R_infty}.

    If $a=k-1$, note that $|M|_{R,k-1}(z)=|M|_R(z)$ from Definition \ref{def:values_of_evaluations}, so we can look for solutions $\lambda\in (k-1,k)\cup[k,\infty)=I_{k-1}$, and in this case the conclusion follows from Proposition \ref{prop:discs}.
    \end{proof}

    \begin{remark}
        Note that the $a$ in the definition \eqref{eq:Hahn_new_disc} of the $\Delta_{S,R,a}$  is necessary, as it is not possible to consider a single polynomial to detect all the Hahn solutions $R\cup\lambda$ with $\lambda\in\mathbb{R}_{\geq0}\setminus\{0,\ldots,k-1\}$ and $R\subset [0,min\{k-1,\lfloor\lambda\rfloor\}]$, since the monomials $|M|_{R,a}(z)=|M_{\leq a}|_R(\infty)+|M_{\geq a+1}|_{\emptyset}(z)$ depend on the choice of $a=0,\ldots,k-1$.
    \end{remark}

We have the following extension of Thm. \ref{thm:crit_fin} for the Hahn case. 
\begin{proposition}
\label{thm:H(P)_finite}
    Let $S$ be a support set. For $a\in[k-1]$ and  $\emptyset\neq R\in \mathcal{P}([a])$, the set $Z_{R,a}=\{P\in \mathbb{Z}_S\::\:V(\Delta_{S,R,a})\cap I_a\text{ is infinite}\}\subset (\mathbb{T}^*)^{\#S}$ has dimension at most $\#S-1$. If we set \begin{equation}
        Y_H(S):=\bigcap_{a\in[k-1]}\bigcap_{\emptyset\neq R\subset[a]}Z_{R,a},
    \end{equation}
    then $\#Sol_k(P,[k-1])<\infty$ if and only if $P\notin Y_H(S)$.
\end{proposition}
\begin{proof}
    Let us consider the  incidence variety $V(\Delta_{S,R,a})\subset(\mathbb{T}^*)^s\times\mathbb{T}^*$. By Prop. \ref{prop:discs_Hahn} we have 

  \begin{equation*}
      \{(P,\lambda)\in\mathbb{Z}_S\times I_a\::\:R\cup \lambda\in HSol_k(P)\}=
        V(\Delta_{S,R,a})\cap(\mathbb{Z}^s\times I_a),
  \end{equation*}
where $I_a\subset \mathbb{T}^*$ is defined as $I_a=(a,a+1)$ if $a=0,\ldots,k-2$, and $I_{k-1}=(k-1,\infty)$.
    
    If we take the first projection $\pi_1:(\mathbb{T}^*)^{\#S}\times \mathbb{T}^*\xrightarrow[]{}(\mathbb{T}^*)^{\#S}$ and $P\in \mathbb{Z}_S$ then $\pi_1^{-1}(P)\cong \mathbb{T}$ is a tropical linear variety of dimension 1 and since $V(\Delta_{S,R,a})$ has dimension $\#S$, there exists $Z'_{R,a}\subset (\mathbb{T}^*)^{\#S}$ of dimension at most $\#S-1$ such that the intersection $V(\Delta_{S,R,a})\cap \pi_1^{-1}(P)$ is infinite; the set $Z'_{R,a}$ consists of the points $P\in\mathbb{Z}_S$ where the restriction $\pi_1|_{V(\Delta_{S,R,a})}$ is not generic, so, by considering $I_a\subset \pi_1^{-1}(P)\cong \mathbb{T}$, the subset $Z_{R,a}=\{P\in Z'_{R,a}\::\:V(\Delta_{S,R,a})\cap I_a\text{ is infinite}\}$ also has dimension at most $\#S-1$. It follows that \begin{equation}
      Sol_k(P,[k-1])=\bigcup_{a\in[k-1]} \bigcup_{\emptyset\neq R\subset [a]}V(\Delta_{S,R,a})\cap I_a,
  \end{equation}
  and $\#\{R\cup\lambda\in Sol_k(P)\::\:\lambda\in I_a\}<\infty$ if and only if  $P\notin Z_{R,a}$, so $\#Sol_k(P,[k-1])<\infty$ if and only if $P\notin Z_{R,a}$ for all pairs $(a,R)$.
\end{proof}

Note that for $a=k-1$, since $I_a=(k-1,\infty)$, the set $Z_{R,k-1}\subset \mathbb{Z}_S$   contains $Y_R=V(\Delta_{S,R}(x_{M_1},\ldots,x_{M_s},z)|_{z=\infty})$, and this does not happen for $a\neq k-1$.

In particular, the same argument from Prop. \ref{prop:ud} works: if $\varphi\in HSol_k(P,[k-1])$, then there should exist $M_j\in S(\varphi)$ such that $Supp(M_j)\not\subset [0,\lfloor\lambda\rfloor]$, and we get that  $\lambda$ gets uniquely determined by the condition  
\begin{equation*}
ord_t(P(\varphi)\!)=a_{M_j}+\alpha(M_j,R)+d(M,R)\lambda,
\end{equation*}
since $d(M,R)\in\mathbb{Z}_{>0}$ by \eqref{eq:def_M_R_infty}.

If we are interested in  Laurent tropical solutions,  only the case $a=k-1$ has to be considered, and  $\{P\in \mathbb{Z}_S\::\:V(\Delta_{S,R})\cap I_{k-1}\text{ is infinite}\}$ gets reduced to $Y_R$, since it is only here that infinite amount of integer points can lie, thus we recover the set $Y(S)=\bigcap_{\emptyset\neq R\subset[k-1]}Y_{R}$  from \eqref{eq:deg_sols}.  Note that the sets $Z_{R,a}\subset (\mathbb{T}^*)^{\#S}$ appearing in $Y_H(S)$ are in general not tropical algebraic, contrary to the case of the $Y_R$ appearing in $Y(S)$, which are the tropical hypersurfaces $Y_R=V(\Delta_{S,R}(x_{M_1},\ldots,x_{M_s},z)|_{z=\infty})$.

In particular, since $\mu(P,[k-1])\subset HSol_k(P,[k-1])$, we get that $\#\mu(P,[k-1])<\infty$ if $P\notin Y_H(S)$.

For polynomials $P\notin Y_H(S)$, the decomposition $HSol_k(P)=HSol_k(P,[k-1])\sqcup HSol_k(P,[k-1]^c)$  satisfies 
    \begin{enumerate}
        \item  \begin{equation}
        \begin{aligned}
            HSol_k(P,[k-1]^c)&\xrightarrow[]{\sim} V(\Delta_S|_P(w)\!)\setminus\{-1,-2,\ldots,-k\}\\
            \lambda&\mapsto \lambda-k
        \end{aligned}
    \end{equation}
    \item $\varphi\in HSol_k(P,[k-1])$ iff $ \emptyset\neq R\subset\{0,\ldots,min\{k-1,\lfloor\lambda\rfloor\}\}$, with $\lambda\in\mathbb{R}_{\geq0}\setminus[k-1]$ varying in a finite set. 
    \end{enumerate}

Recall that if $S$ is homogeneous, then $Sing(S)=V(\Delta_S)$ for $\Delta_S$ given in \eqref{eq:disc_homogeneous}. See Corollary \ref{prop:hcase}. If $X_H(S):=Y_H(S)\cap V(\Delta_S(x_1,\ldots,x_n,w)|_{w=0})$, we can extend the concept of regular differential polynomial from Section \ref{sec:ordatmostkmo}.

\begin{definition}
    Let $S$ be homogeneous. We say that $P\in\mathbb{Z}_S$ is Hahn regular if $P\notin X_H(S)$.
\end{definition}

Again, \cite[Prop. 5]{GG26} extends to this context in the following sense.

\begin{proposition}
\label{prop:Shom}
    Let $S$ be homogeneous. $P\in\mathbb{Z}_S$ is Hahn regular  if and only if $\mu(HSol(P)\!)=\mu(P,[k-1])$ consists only of solutions $\varphi=R\cup\lambda$, with $ \emptyset\neq R\subset\{0,\ldots,min\{k-1,\lfloor\lambda\rfloor\}\}$ and $\lambda\in\mathbb{R}_{\geq0}\setminus\{0,\ldots,k-1\}$ varying in a finite set. 
\end{proposition} 
\begin{proof}
    if $S$ is homogeneous of degree $d$, then  $S_{min }=S_{max} =\{d\}$. In particular 

\begin{equation*}
    \mu(P,[k-1]^c)=\begin{cases}
        \emptyset, &\text{ if }P\notin V(\Delta_S(x_{M_1},\ldots,x_{M_s},z)|_{z=0}),\\
        \mathbb{R}\setminus\{0,\ldots,k-1\},&\text{ otherwise},\\
    \end{cases}
\end{equation*}
by Remark \ref{rem:coro_2}.
\end{proof}

Thus, the new tropical Hahn phenomena at the level of minimal solutions for the homogeneous case is the appearance of solutions $ \emptyset\neq R\subset\{0,\ldots,min\{k-1,\lfloor\lambda\rfloor\}\}$, with $\lambda\in\mathbb{R}_{\geq0}\setminus[k-1]$.

We haven't studied yet the properties of $HSol_k(P,[k-1])^c\subset HSol_k(P)$ for the case in which $S$ is not homogeneous, we will do this in Sect. \ref{sec:ordatleastk}.

\subsection{Minimal solutions $ord(\varphi)\notin\{[ord(P)-1]\}$ revisited}
\label{sec:ordatleastk}
In this section we assume that $S$ is a support set which is not homogeneous (otherwise, see Prop. \ref{prop:Shom}). Let $P\in\mathbb{Z}_S$, recall that the set $\mu(P,[k-1]^c)\subset \mu(HSol(P)\!)$ can be described by the bend locus $V(\Delta_S|_P(w)\!)\setminus[k-1]\subset\mathbb{R}\setminus[k-1]$ of an expression $\Delta_S|_P(w):\mathbb{T}^*\xrightarrow[]{}\mathbb{T}^*$, with $\Delta_{S,\emptyset}|_P(w):=min_j\{|M_j|_k+a_{M_j}+deg(M_j)w\}$. Moreover, recall from Theorem \ref{thm:crit_fin_mero} that $\#\mu(LSol(P)\!)<\infty$ if and only if $P\notin V(\Delta_{S_{min},\emptyset}|_{w=0},\Delta_{S_{max},\emptyset}|_{w=0})$, where $S_{min}$ and $S_{max}$ denote the set of minimal and maximal degree monomials, respectively. 

In a similar fashion, we can try to characterize the locus $HSing(S)=\{P\in\mathbb{Z}_S\::\:|\mu(HSol(P)\!)|=\infty\}$. Let $deg(S)=\{d_1<\cdots<d_n\}$. If $\#deg(S)=2$, then  $HSing(S)=LSing(S)$, since $V(\Delta_S|_P(w)\!)\setminus\{[k-1]\}$ only exhibits  {\it asymptotic behavior} ($w>>0$ or $w<<0$).

    However, if $\#deg(S)>2$, then there may be  {\it whole intervals} $I\subset \mathbb{R}$ contained in $V(\Delta_S|_P(w)\!)$ coming from the bounded segments of the graph of the evaluation function $\Delta_S|_P(w):\mathbb{T}^*\xrightarrow[]{}\mathbb{T}^*$; these segments do not cause any problem with minimal Laurent solutions $\mu(P,[k-1]^c)\subset \mu(LSol(P)\!)$ since $I\cap\mathbb{Z}$ remains finite. This is no longer true for $\Gamma=\mathbb{Q}$ or $\Gamma=\mathbb{R}$. In this section we study the bend locus $V(\Delta_S|_P(w)\!)$  of $\Delta_S|_P(w):\mathbb{T}^*\xrightarrow[]{}\mathbb{T}^*$.

    What can we say about the finite set $V(\Delta_{S,\emptyset}|_P(z)\!)\cap\mathbb{Z}$? Can we bound its cardinality?

If $S$ is a support set, we denote by $[|deg(M)\::\:M\in S|]$ its degree multiset (or list, this is, a collection that admits repetitions). 

\begin{proposition}
Let $S$ be as usual. The set 
\begin{equation}
    \{\Delta_S|_P(w)\::\:P\in \mathbb{Z}_S\}
\end{equation}
does not depend on $S$, and depends only on the multiset 
$[|deg(S)|]=[|deg(M)\::\:M\in S|]$.
\end{proposition}
\begin{proof}
    We have $\Delta_S|_P(w)=min_j\{|M_j|_k+a_{M_j}+deg(M_j)w\}$, with $|M_j|_k\in\mathbb{Z}$. We denote by $\mathbb{TS}_{[|deg(S)|]}[w]$ the set of all formal expressions
    \begin{equation}
        \sigma(w)= min_{M_j\in S}\{\alpha_{j}+deg(M_j)w\}, \quad\alpha_{j}\in\mathbb{Z}.
    \end{equation}
    
We have that the next assignment 
\begin{equation}
    \begin{aligned}
       \Delta_S|_{-}(w): \mathbb{Z}_S&\xrightarrow[]{}\mathbb{TS}_{[|deg(S)|]}[w]\\
        P&\mapsto \Delta_S|_P(w)
    \end{aligned}
\end{equation}
is bijective (this follows from the fact that $\overline{\mathbb{Z}}$ is a semifield ); we have $\alpha_j(P)=|M_j|_S+a_{M_j}$ for all $j=1,\ldots,s$, if $\alpha_j(P)=\alpha_j(Q)$ for $P,Q\in \mathbb{Z}_S$, then $|M_j|_S+a_{M_j}=|M_j|_S+b_{M_j}$ for all $j$, so $P=Q$. Now let $\sigma= min_j\{\alpha_{j}+deg(M_j)w\}$, we see that $P=(a_{M_j})=(\alpha_j(P)-|M_j|_S)$ satisfies $\Delta_S|_{P}(w)=\sigma$.
\end{proof}

By our hypotheses on support sets (see Definition \ref{def:supsetk}), we have $deg(M)>0$ for all $M\in S$, so we can codify the multisets $[|deg(M)\::\:M\in S|]$ into integer partitions: sequences $\lambda=(\lambda_s\geq\cdots \geq\lambda_1)$  of $s>0$ positive integers, which are said to be a partition of the integer $\sum_i\lambda_i$.

\begin{definition}
    A tropical partition-sum is a formal expression 
    \begin{equation}
        \sigma(w)= \bigoplus_{i=1}^s a_i\odot w^{\odot \lambda_i},\quad a_i\in\mathbb{Z},
    \end{equation}
    where $\lambda=(\lambda_s\geq\cdots \geq\lambda_1)$ is a partition and  $w$ is a variable.
 We denote by $\mathbb{TPS}[w]$ the set of all tropical partition-sums.
    \end{definition}

We see tropical partition-sums as generalizations of tropical polynomials, so we can consider the induced evaluation function $ev_\sigma:\mathbb{T}^*\xrightarrow[]{}\mathbb{T}^*$  given by $ev_\sigma(q)=min_{i=1,\ldots,s}\{a_i+\lambda_i q\}$, and by $V(\sigma)=\{q\in \mathbb{T}^*\::\:ev_\sigma(q)\text{``} = \text{''}0_{\mathbb{T}}\}$ its bend locus. 

\begin{remark}
Given a partition $\lambda=(\lambda_s\geq\cdots \geq\lambda_1)$, we denote by $(\mathbb{T}^*)^s\cong\mathbb{TPS}_\lambda[x]\subset \mathbb{TPS}[x]$ the set of all tropical partition-sums having support $\lambda$. 

In principle, if $[|deg(S)|]=\lambda$, then we can use this set to study the behavior of the sets $\mu(P,[k-1]^c)\subset HSol(P)$ (and also $\mu(P,[k-1]^c)\subset LSol(P)$) for $P\in\mathbb{Z}_S$. However, there is ambiguity in this procedure since  $k=\max\{ord(M)\::\:M\in S\}$ can't be recovered from $\lambda$.
\end{remark}

\begin{convention}
    From now on we leave this ambiguity behind, and focus ourselves on the study of the properties of the tropical partition-sums; we will now use the variable $x$ instead of $w$.
\end{convention}

    \begin{example}We  give some extreme examples of objects $\sigma(x)= \bigoplus_{i=1}^s a_i\odot x^{\odot \lambda_i}$.
    \begin{enumerate}
        \item (a monomial) if $s=1$, then $\sigma(x)=a\odot x^{d}$ is a single monomial, so $V(\sigma)=\{\infty\}$. For this reason, we frequently suppose that $s>1$
        \item (homogeneous) if $s>1$, but $\lambda_i=\lambda_j=d$ for all $i,j$, then $\sigma(x)=\bigoplus_ia_i\odot x^{d}$ satisfies 
        \begin{equation*}
            V(\sigma)=\begin{cases}
                \{\infty\},&\text{ if }\bigoplus_{i=1}^sa_i\text{ does not vanish tropically},\\
                \mathbb{T},&\text{otherwise}.
            \end{cases}
        \end{equation*}
        for this reason, we  frequently suppose that $s>1$ and $\lambda_s>\lambda_1$.
        \item (all different) if $s>1$ and $\lambda=(\lambda_s>\cdots >\lambda_1)$, then $\sigma(x)$ is just a tropical polynomial with coefficients in $\overline{\mathbb{Z}}$ and $Supp(f)=\lambda$. In particular $V(\sigma)$ is finite.
    \end{enumerate}
\end{example}

The previous examples motivate the next definitions. If $\lambda=(\lambda_s\geq\cdots \geq\lambda_1)$ is a partition, we denote by $deg(\lambda)$ its associated set.

\begin{definition}
    Let $\sigma=\bigoplus_i a_i\odot x^{\lambda_i}$ be a tropical partition-sum. We define its polynomialization $\Pi(\sigma)\in\overline{\mathbb{Z}}[x]$ as \begin{equation}
       \Pi(\sigma)(x):=\bigoplus_{d\in deg(\lambda)}[\bigoplus_{\{\lambda_i=d\}}a_i]\odot x^{d},  
    \end{equation}
this is, $\Pi(\sigma)(x):=\bigoplus_{j} b_j\odot x^{j}$, where $b_j=\bigoplus_{\{\lambda_i=j\}}a_i.$
\end{definition}

Clearly the image of the map $\Pi:\mathbb{TPS}[x]\xrightarrow[]{}\overline{\mathbb{Z}}[x]$ is the set of nonzero polynomials with no constant term, this is,  $\langle x\rangle\setminus\{0\}$.

\begin{definition}
    We say that $\sigma=\bigoplus_i a_i\odot x^{\lambda_i}$ is of polynomial type if $V(\sigma)=V(\Pi(\sigma)\!)$, equivalently, if $V(\sigma)$ is finite.
\end{definition}

It is hard to describe algebraically the locus of tropical partition-sums which are of polynomial type. We have the following result. 
\begin{proposition}
    Let $\sigma=\bigoplus_i a_i\odot x^{\lambda_i}$ be a tropical partition-sum, and let $\Pi(\sigma)(x)=\bigoplus_j b_j\odot x^{d_j}$. If $b_j=\bigoplus_{\{\lambda_i=d_j\}}a_i$ does not vanish tropically for all $d_j\in deg(\lambda)$, then $V(\sigma)=V(\Pi(\sigma)\!)$.
\end{proposition}
\begin{proof}
We write $\Pi(\sigma)(x)=b_1\odot x\oplus\cdots \oplus b_d\odot x^{\odot d}$ with $b_i\in\overline{\mathbb{Z}}$. Since $\Pi(\sigma)$, is a polynomial of degree $d\geq1$ with no constant term, the bend locus $V(\Pi(\sigma)\!)\subset\mathbb{T}$ consists of $d$ roots  in $\mathbb{T}$ (counting multiplicities), say $\{-\infty<x_1\leq\cdots\leq x_d=\infty\}$.

We know that the bend locus $V(\Pi(\sigma)\!)\subset\mathbb{T}$ depends only on the class $[ev_{\Pi(\sigma)}]\in\mathcal{O}_{\mathbb{T}}(\mathbb{T})$ of its evaluation function $ev_{\Pi(\sigma)}$, and this class is represented by another tropical polynomial $Q(x)=b_1'\odot x\oplus\cdots \oplus b_d'\odot x^{\odot d}$ whose support is the set of relevant degrees of $\Pi(\sigma)$, so $Supp(Q)\subset Supp(\Pi(x)\!)$ (note that the smallest and the biggest degrees are always relevant).

We set $x_0=-\infty$, and consider the graph $\Gamma\subset \mathbb{T}^2$ of   $ev_{Q}$, which is made out of line segments over the interval $[x_i,x_{i+1}]\subset\mathbb{R}$ for  $i=0,\ldots,d-1$ (which may be reduced to a single point), and it is described by one of the monomials $b_{j(i)}\odot x^{d_{j(i)}}$ of $f_\sigma(x)$. If $i=0$, we have that $d_{j(i)}= deg(f_\sigma)$, and if $i=d-1$, we have that $d_{j(i)}=min deg(f_\sigma)$, and if $f_\sigma$ has no more monomials, we recover our previous condition. 

If $f_\sigma(x)$ has at least three monomials, then the graph $\Gamma$ may or may not contain  inner line segments, and if it does, it is no longer clear which monomial describes it ($f_\sigma(x)$ may have be  irrelevant monomials). But if none of the expressions $b_j=\bigoplus_{\{\lambda_i=d_j\}}a_i$ vanish tropically, then we can ensure that none of the line segments of $\Gamma$ belongs to $V(\sigma)$, since  both $\sigma$ and $f_\sigma$ define the same evaluation function, so the graph of $ev_{\sigma}$ depends only on the relevant monomials of $f_\sigma$, see Definition \ref{def:rel_mon}.
\end{proof}

In a similar fashion, if $S$ is not homogeneous, we say that $P\in\mathbb{Z}_S$ is Hahn regular if $\#\mu(P,[k-1]^c)<\infty$ and $P\notin Y_H(S)$. We will finish this section by describing a closed set $Z\supset HSing(S)$.

\begin{corollary}
    Let $S$ be not homogeneous. Let \begin{equation}
        \Omega(S):=V(\Delta_{S_{d},\emptyset}|_{w=0},\:d\in deg(S)\!)\subset\mathbb{T}^s,
    \end{equation}
    where $S_d=\{M\in S\::\:deg(M)=d\}$.   If $P\notin\Omega$, then $\#\mu(P,[k-1]^c)<\infty$.
 
\end{corollary}

 Given $S$ and  $P\in\mathbb{Z}_S$, we follow the next steps to compute $\mu(P,[k-1]^c)$:
\begin{enumerate}
    \item compute $\sigma(P):=\Delta_S|_P(w)=\bigoplus_ja_j\odot w^{\odot \lambda_j}$,
    \item compute the set of relevant slopes (monomials) of the polynomialization $f_{\sigma(P)}=\bigoplus_jb_j\odot w^{\odot d_j}$, i.e., the slopes appearing on the graph of the evaluation function $ev_{f_{\sigma(P)}}=ev_\sigma$ (the smallest and biggest slopes of $S$ are always relevant),
    \item for every relevant slope $d_k$ ,  check  wether or not $b_k=\bigoplus_{\{\lambda_i=d_k\}}a_i$ vanishes. 
\end{enumerate}

From this procedure, it looks hard to determine $HSing(S)$ explicitly for a fixed $S$. An alternative condition was introduced in \cite{GG26}, namely that of each factor $a_i$ appearing in $b_k=\bigoplus_{\{\lambda_i=d_k\}}a_i$ for all $k$ to be different. This is an open condition called non weakly tropical vanishing.

\begin{remark}
For a support set $S$, $a\in[k-1]$ and $\emptyset\neq R\subset[a]$ fixed, one can also consider a specialization    for $P=(a_M)\in\mathbb{Z}_S$: $$\Delta_{S,R,a}|_P(z)=\bigoplus_{M\in S}|M|_{R,a}(z)\odot a_M=min_{M}\{|M_{\leq a}|_R(\infty)+a_M+|M_{\geq a+1}|_{\emptyset}(z)\}.$$

The only differences with the case $R=\emptyset$ discussed here are that $\Delta_{S,R,a}|_P(z)$ may have constant (degree 0) monomials, in fact, $|M|_{R,a}(z)\odot a_M$ is constant if and only if $M=M_{\leq a}$. Also, we have to look for solutions  $V(\Delta_{S,R,a}|_P(z)\!)\subset I_a$.
\end{remark}

\subsection{The set of $t$-adic orders of formal solutions of an \textsc{oade} at the origin} 
\label{sect:appPPord}

We now consider applications of the previous set of ideas towards the study of  the set $S(\mathfrak{P},0):=ord_t(HSol(\mathfrak{P})\!)\subset \mathbb{R}$ from  \eqref{eq:OCDE_tadicord}, where $\mathfrak{P}\in\mathbb{C}[[t]]\{y\}$ is an \textsc{oade}. We start by recalling the following important result, which yields the concrete connection between formal power series solutions of \textsc{oade}s, and tropical solutions of their tropicalization (for completeness, we give a proof in the context of formal Hahn solutions).

\begin{proposition}
\label{prop:trop_Hahn}
Let $\mathfrak{P}= \sum_{M_i\in S} {\mathcal A}_{M_i} E_{M_i}$, and  $HSol(\mathfrak{P})=\{\Phi\in \mathbb{C}[\![t^{\mathbb{R}}]\!]\::\:\mathfrak{P}(\Phi)=0\}$. If $P=trop(\mathfrak{P})$, then we have a map 
    \begin{equation}
        Supp:HSol(\mathfrak{P})\xrightarrow[]{}HSol(P).
    \end{equation}
\end{proposition}
\begin{proof}
The proof is an easy consequence of the definitions. Let $\Phi\in \mathbb{C}[\![t^{\mathbb{R}}]\!]$, we write $i.t.(\Phi)=a_{ord_t(\Phi)}t^{ord_t(\Phi)}$ the initial term of $\Phi$, with $a_{ord_t(\Phi)}\in\mathbb{C}^*$, so we can write $\Phi=i.t.(\Phi)+\widetilde{\Phi}$, with $ord_t(\widetilde{\Phi})>ord_t(\Phi)$.

Let $\mathfrak{P}= \sum_{M_i\in S} {\mathcal A}_{M_i} E_{M_i}$.
We start with the classical initial (indicial) equation $in_{\Phi}(\mathfrak{P})$. For every monomial ${\mathcal A}_{M_i} E_{M_i}$ of $\mathfrak{P}$, we have $ord_t(\mathcal{A}_{M_i} E_{M_i}(\Phi)\!)=a_{M_i}+ord_t(E_{M_i}(\Phi)\!)$. Thus, the quantity $trop_{\Phi}(\mathfrak{P}):=min_{M_i}\{a_{M_i}+ord_t(E_M(\Phi)\!)\}$ gives us the minimal $t$-adic order of the evaluations $\{\mathcal{A}_{M_i} E_{M_i}(\Phi)=\sum_{j\geq j_0(i)}a_{i,j}t^j\}_i$. In the sum 

$$\mathfrak{P}(\Phi)=\sum_{M_i\in S}\sum_{j\geq j_0(i)}a_{i,j}t^j=\alpha t^{trop_{\Phi}(\mathfrak{P})}+\widetilde{\mathfrak{P}(\Phi)}$$

the only terms that can contribute to the initial term $\alpha t^{trop_{\Phi}(\mathfrak{P})}$ are the initial terms $i.t.(\mathcal{A}_{M_i} E_{M_i}(\Phi)\!)$ of the monomials $\mathcal{A}_{M_i} E_{M_i}$ having the right valuation $trop_{\Phi}(\mathfrak{P})$, thus 

$$\alpha t^{trop_{\Phi}(\mathfrak{P})}=\sum_{M_i\::\:trop_{\Phi}(\mathfrak{P})=a_{M_i}+ord_t(E_M(\Phi)\!)}i.t.(\mathcal{A}_{M_i} E_{M_i}(\Phi)\!)$$

If $\mathfrak{P}(\Phi)=0$, then $\alpha t^{\trop_\Phi(P)}=0$, so there must exist at least two different monomials $M_j,M_k$ such that $$trop_{\Phi}(\mathfrak{P})=min_{M_i}\{a_{M_i}+ord_t(E_M(\Phi)\!)\}=a_{Mji}+ord_t(E_{M_j}(\Phi)\!)=a_{M_k}+ord_t(E_{M_k}(\Phi)\!)$$

To finish, recall that the map $ord_t:\mathbb{C}[\![t^{\mathbb{R}}]\!]\xrightarrow[]{}\mathbb{T}$ factors through the support map $Supp:\mathbb{C}[\![t^{\mathbb{R}}]\!]\xrightarrow[]{}\mathbb{B}[\![t^{\mathbb{R}}]\!]$ through a map $ord_t:\mathbb{B}[\![t^{\mathbb{R}}]\!]\xrightarrow[]{}\mathbb{T}$ (see \eqref{diag:relevant}). 

    This means that if $\Phi\in \mathbb{C}[\![t^{\mathbb{R}}]\!]$ and $\phi=Supp(\Phi)$, then $$ord_t(\tfrac{d^i\Phi}{dt^i})=ord_t(\tfrac{d^iSupp(\Phi)}{dt^i})=ord_t(\tfrac{d^i\phi}{dt^i}),$$
    and this extends to $ord_t(E_{M}(\Phi)=ord_t(E_{M}(\phi)$ for all $M$. Since $P=trop(\mathfrak{P})$, we have that $trop_\phi(P)\text{``} = \text{''}0_{\mathbb{T}}$, this is, this implies $\phi=Supp(\Phi)\in HSol(P)$.
\end{proof}

 If $P=trop(\mathfrak{P})$, the map $Supp:HSol(\mathfrak{P})\xrightarrow[]{}HSol(P)$ from Prop. \ref{prop:trop_Hahn} is not surjective, this is, there may be tropical solutions $\varphi\in Sol(P)$ which do not lift to a classical solution $\Phi\in Sol(\mathfrak{P})$. 

We are now ready to give our main result, which says that the set $S(\mathfrak{P},0)$ can be computed from the elementary Hahn $k$-solutions of the tropicalization $P=trop(\mathfrak{P})$.

    \begin{theorem}
    \label{thm:app_tropical_Painleve}
        If $\mathfrak{P}$ is an \textsc{oade} and $P=trop(\mathfrak{P})$, then $S(\mathfrak{P},0)\subset ord_t(HSol_k(P))$.
    \end{theorem}
    \begin{proof}

    Recall that we have a map $\pi_k:HSol(P)\xrightarrow[]{}HSol_k(P)$ satisfying   $ord_t(\varphi)=ord_t(\pi_k(\varphi)\!)$ for all 
    $\varphi\in HSol(P)$, and since the map $Supp:HSol(\mathfrak{P})\xrightarrow[]{}HSol(P)$  from Prop. \ref{prop:trop_Hahn} also commutes with $ord_t$ by \eqref{diag:relevant}, we have a commutative diagram 
    \begin{equation}
    \label{diag:factorization}
\xymatrix{HSol(\mathfrak{P})\ar[r]^{\pi_k\circ Supp}\ar[dr]_{ord_t}&HSol_k(P)\ar[d]^{ord_t}\\
        &\mathbb{T}}\qedhere
    \end{equation}
    \end{proof}

    We have seen that there are two relevant sets of polynomial solutions for an \textsc{otde} $P$, namely $\mu(HSol(P)\!)\subset HSol_k(P)$, together with decompositions $\mu(HSol(P)\!)=\mu(P,[k-1])\sqcup\mu(P,[k-1]^c)$ and $HSol_k(P)=HSol_k(P,[k-1])\sqcup HSol_k(P,[k-1]^c)$, and since $HSol_k(P,[k-1]^c)=\mu(P,[k-1]^c)$, we have 
\begin{equation*}
    HSol_k(P)\setminus \mu(Sol(P)\!)=HSol_k(P,[k-1])\setminus\mu(P,[k-1]).
\end{equation*}

This says that $\mu(Sol(P)\!)\subset HSol_k(P)$ only differ at the level of $Sol_k(P,[k-1])\setminus\mu(P,[k-1])$, and $Sol_k(P,[k-1])$ is finite if and only if $P\notin Y_H(S)$ (where $S=Supp(P)$), by Proposition \ref{thm:H(P)_finite}. This leads us to the study of the chain $$\mu(P,[k-1])\subseteq HSol_k(P,[k-1])$$ of minimal and elementary Hahn $k$-solutions having order in $[k-1]$. Now, if $\varphi\in HSol_k(P,[k-1])$ there is at least one $\psi\in \mu(HSol(P)\!)$ such that $\psi\subset \varphi$, and $ord(\psi)\geq ord(\varphi)$, but it may happen that $ord(\psi)> ord(\varphi)$ for every minimal solution $\psi\subset \varphi$. See Example \ref{ex_Grigoriev}.

 If this does not happen, then we say that $P$ satisfies the {coherent order property}.
\begin{definition}
    We say that $P$ satisfies the {coherent order property} if for any $\varphi\in  HSol_k(P,[k-1])$, there is a unique $\varphi_{0}\in \mu(P,[k-1])$ such that $\varphi_{0}\subset \varphi$ and $ ord_t(\varphi)=ord_t(\varphi_0)$.
\end{definition}

The following result tells us that if $P$ has the {coherent order property}, then $S(\mathfrak{P},0)\subset ord_t(\mu(HSol(P)\!))$.

\begin{proposition}
    If $P$ satisfies the {coherent order property} and $trop(\mathfrak{P})=P$, then we have commutative diagram  
    \begin{equation}
        \xymatrix{HSol(\mathfrak{P})\ar[d]_{ord_t}\ar[dr]&\\
        \mathbb{T}&\mu(HSol(P)\!)\ar[l]^-{ord_t}}
    \end{equation}
\end{proposition}
\begin{proof}
We want to define a map $min: HSol_k(P)\xrightarrow[]{}\mu(Sol(P)\!)$ such that $ord_t(min(\varphi)\!)=ord_t(\varphi)$. For $\varphi\in HSol_k(P)$ such that $ord_t(\varphi)\in[k-1]$, if $P$ satisfies the {coherent order property}, then the assignment $\varphi\mapsto\varphi_0$ defines a map $min:HSol_k(P)\xrightarrow[]{}\mu(HSol(P)\!)$, so we have a diagram \begin{equation}
    \label{diag:factorization:2}
\xymatrix{HSol(\mathfrak{P})\ar[r]^{\pi_k\circ Supp}\ar[d]_{ord_t}\ar[dr]&HSol_k(P)\ar[d]^{min}\\
        \mathbb{T}&\mu(HSol(P)\!)\ar[l]^-{ord_t}}
    \end{equation}
    and this finishes the proof.
\end{proof}

In the linear case,  if $P$ has order $k$ and it is generic, then $\mu(Sol(P)\!)=\{t^p+t^{q(p)}\}_{p\in[k-1]}$ with $\{q(p)\geq k\}_{p\in[k-1]}$ all distinct. Then $P$ satisfies the {coherent order property} if and only if $q(0)<q(1)<\cdots <q(k-1)$. Indeed, if $q(p')<q(p)$ for some $p<p'$, then the  elementary $k$-solution corresponding to $\varphi(t)=t^p+t^{p'}+t^{q(p')}+t^{q(p)}$ is $t^p+t^{p'}+t^{q(p')}$, which contains the minimal solution $t^{p'}+t^{q(p')}$, with $p<p'$.

\begin{example}
\label{ex_Grigoriev}

The following concrete example was found by D. Grigoriev. Let $P=t^5y_0+ t^2y_2+ y_3$. By computing $ord_tP(t^i)$ for $i=0,1,2,3,5,8,9$: 
\begin{table}[!htb]
    \centering
    \begin{tabular}{c|c}
    $i$&  $ord_t(P(t^i)\!)$\\\hline
    0 & $5\oplus\infty\oplus\infty$\\
    1 & $6\oplus\infty\oplus\infty$\\
    2 & $7\oplus2\oplus\infty$\\
    3 & $8\oplus3\oplus0$\\
    \end{tabular}
    \quad\quad
    \begin{tabular}{c|c}
    $i$&  $ord_t(P(t^i)\!)$\\\hline
    5 & $10\oplus5\oplus2$\\
    8 & $13\oplus8\oplus5$\\
    9 & $14\oplus9\oplus6$\\
    \end{tabular}
    \quad  \begin{tabular}{c|c}
    $\varphi$&  $ord_t(P(\varphi)\!)$\\\hline
    $1+t^8$ & $5\oplus8\oplus5$\\
    $t+t^9$ & $6\oplus9\oplus6$\\
    $t^2+t^5$ & $7\oplus2\oplus2$\\
    \end{tabular}
    \caption{Computations for $P=t^5y_0+ t^2y_2+ y_3$.}
    \label{tab:placeholder}
\end{table}
thus we see that $P$ is generic, and that $\mu(Sol(P)=\{1+t^8, t+t^9, t^2+t^5\}$. Thus the elementary $3$-solution associated to  $t+t^2+t^5+t^9$ is $t+t^2+t^5$, which is not minimal, but contains the minimal solution $t^2+t^5$.
\end{example}

Let $S$ be a support set. We finish with a procedure to check if $HSol_k(P)=HSol_k(P,[k-1])\sqcup HSol_k(P,[k-1]^c)$ is finite or not:
\begin{enumerate}
\item  we check that $P\notin Y(S)$ from \eqref{eq:deg_sols}, for which we need at most $2^k-1=\#[\mathcal{P}([k-1])\setminus\emptyset]$ verifications. 
\item Once that $P\notin Y(S)$, all the members of $HSol_k(P,[k-1])$ are of the form $(R,\lambda)$, with $\lambda\in I_a$ for $a\in[k-1]$ and $\emptyset\neq R\subset [a]$, and we need to verify that $P\notin Z_{R,a}$ from prop. \ref{thm:H(P)_finite}, which leaves us with $(2-1)+\cdots+(2^{k}-1)=2^{k+1}-k-2$ verifications.
    \item Finally, a procedure to compute $HSol_k(P,[k-1]^c)=\mu(P,[k-1]^c)$ was given 
    \begin{enumerate}
        \item at the end of Section \ref{sec:ths}, if $S$ is  homogeneous.
        \item at the end of Section \ref{sec:ordatleastk}, if $S$ is not homogeneous.
    \end{enumerate}
    
\end{enumerate}

\addtocontents{toc}{\setcounter{tocdepth}{-1}} 
\section*{Acknowledgments}
\addtocontents{toc}{\setcounter{tocdepth}{3}}

The authors wish to thank D. Grigoriev and P. Rubí Pantaleón for valuable discussions during the elaboration of this paper.

\end{document}